\PassOptionsToPackage{dvipsnames}{xcolor}
\documentclass[reqno]{amsart}

\usepackage{amsmath,amsthm,amssymb,bm}
\usepackage{hyperref}
\usepackage{a4wide}
\usepackage{cleveref}
\usepackage{graphicx,color}
\usepackage{tikz}
\usepackage{ytableau}
\usepackage{arydshln}
\usepackage[dvipsnames]{xcolor}
\usepackage{etoolbox}
\usepackage{breqn}
\usepackage{nicematrix}
\apptocmd{\sloppy}{\hbadness 10000\relax}{}{}

\usetikzlibrary{positioning,decorations.pathmorphing, decorations.pathreplacing, decorations.shapes}

\numberwithin{equation}{section}

\newtheorem{thm}{Theorem}[section]
\newtheorem{lem}[thm]{Lemma}
\newtheorem{prop}[thm]{Proposition}
\newtheorem{cor}[thm]{Corollary}
\theoremstyle{definition}
\newtheorem{remark}[thm]{Remark}
\newtheorem{exam}[thm]{Example}
\newtheorem{defn}[thm]{Definition}
\newtheorem{conj}[thm]{Conjecture}

\crefname{lem}{Lemma}{Lemmas}
\crefname{thm}{Theorem}{Theorems}
\crefname{prop}{Proposition}{Propositions}
\crefname{question}{Question}{Questions}
\crefname{defn}{Definition}{Definitions}
\crefname{conj}{Conjecture}{Conjectures}
\crefname{figure}{Figure}{Figures}
\crefname{cor}{Corollary}{Corollaries} 
\crefformat{equation}{(#2#1#3)}
\Crefformat{equation}{Equation #2(#1)#3}

\AddToHook{env/lem/begin}{\crefalias{thm}{lem}}
\AddToHook{env/prop/begin}{\crefalias{thm}{prop}}
\AddToHook{env/cor/begin}{\crefalias{thm}{cor}}
\AddToHook{env/exam/begin}{\crefalias{thm}{exam}}
\AddToHook{env/defn/begin}{\crefalias{thm}{defn}}
\AddToHook{env/conj/begin}{\crefalias{thm}{conj}}
\AddToHook{env/question/begin}{\crefalias{thm}{question}}
\AddToHook{env/problem/begin}{\crefalias{thm}{problem}}
\AddToHook{env/remark/begin}{\crefalias{thm}{remark}}

\newcommand\sgn{\operatorname{sgn}}
\newcommand\NN{\mathbb{N}}

\newcommand{\CC}{\mathbb{C}}
\newcommand{\ZZ}{\mathbb{Z}}

\newcommand\sn{\mathfrak{S}_n}
\newcommand\JT{\operatorname{JT}}
\newcommand\tJT{\mathcal{E}}
\newcommand\HJT{\mathcal{H}}
\newcommand\Par{\operatorname{Par}}
\newcommand\SSYT{\operatorname{SSYT}}
\newcommand\SYT{\operatorname{SYT}}
\newcommand\Des{\operatorname{Des}}
\newcommand\NDes{\operatorname{NDes}}
\newcommand\comp{\operatorname{comp}}
\newcommand\Frob{\operatorname{Frob}}
\newcommand\imm{\operatorname{imm}}

\newcommand\wt{\operatorname{wt}}
\newcommand\type{\operatorname{type}}
\newcommand\TL{\operatorname{TL}}
\newcommand\s{\mathfrak{s}}
\renewcommand\vec[1]{\mathbf{#1}}

\makeatletter
\newcommand{\ostar}{\mathbin{\mathpalette\make@circled\star}}
\newcommand{\make@circled}[2]{%
  \ooalign{$\m@th#1\smallbigcirc{#1}$\cr\hidewidth$\m@th#1#2$\hidewidth\cr}%
}
\newcommand{\smallbigcirc}[1]{%
  \vcenter{\hbox{\scalebox{0.77778}{$\m@th#1\bigcirc$}}}%
}

\newcommand\T[1]{\tau(#1)}
\newcommand\B[1]{\mathfrak{p}(#1)}

\newcommand{\csn}{\mathbb{C}[\sn]}
\newcommand{\K}{\mathcal{B}}

\title{Hadamard products of dual Jacobi--Trudi matrices}

\author{Robert Angarone}
\address{Department of Mathematics \\ University of Minnesota \\ Minneapolis \\ United States}
\email{angar017@umn.edu}

\author{Jang Soo Kim}
\address{Department of Mathematics \\ Sungkyunkwan University \\ Suwon \\ Korea}
\email{jangsookim@skku.edu}

\author{Jaeseong Oh}
\address{Department of Mathematics \\ Sungkyunkwan University \\ Suwon \\ Korea}
\email{jaeseongoh@skku.edu}

\author{Daniel Soskin}
\address{Department of Mathematics \\ University of California \\ Los Angeles \\ United States}
\email{dsoskin@math.ucla.edu}

\begin{document}

\begin{abstract}
  We study positivity properties of Hadamard products of Jacobi--Trudi matrices. Mal\'{o} proved that the Hadamard (entrywise) product of two totally positive upper-triangular Toeplitz matrices whose Toeplitz sequences are the coefficient sequences of real-rooted polynomials with nonpositive zeros is again totally positive. Sokal conjectured that this result can be strengthened to total monomial positivity for the Hadamard product of Jacobi--Trudi matrices. In this paper we show that Temperley--Lieb immanants are Schur positive for Hadamard products of Jacobi--Trudi matrices given by ribbon-like skew shapes. In particular, we affirm Sokal's conjecture for minors given by ribbon-like skew shapes. Moreover, we provide a manifestly positive Schur expansion for Temperley--Lieb immanants evaluated on the Hadamard product of Jacobi--Trudi matrices indexed by ribbons. In addition, for the ribbon case, we construct a corresponding representation, offering a representation-theoretic proof of the Schur positivity.
\end{abstract}

\maketitle

\section{Introduction}
A matrix is said to be \emph{totally positive} if all of its minors are nonnegative. Total positivity arises in various areas of mathematics and has wide-ranging applications. Originally studied by Gantmacher and Krein \cite{Gantmacher1935,Gantmakher1937} in the context of classical analysis and numerical interpolation, total positivity has since become a central topic in representation theory, geometry, mathematical physics, and algebraic combinatorics (see the survey by Fomin \cite{Fomin2010} and references therein).

By the Cauchy--Binet theorem, it is immediate that the product
\( MN \) of two totally positive matrices \( M \) and \( N \) is also
totally positive. However, their Hadamard (entrywise) product
\( M * N \) need not be totally positive in general. 
Nevertheless, there exists a class of matrices whose total positivity is preserved under the Hadamard product. 
A \textit{Toeplitz matrix} with \textit{Toeplitz sequence $(a_0,a_1,\ldots)$} is an infinite, upper-triangular matrix
of the form $(a_{j-i})_{i,j=0}^\infty$.
The \emph{Laguerre--Pólya class} is the class of entire functions that arise as uniform limits of
univariate polynomials with nonnegative coefficients and real roots. 
Mal\'{o} \cite{Malo1895} proved that the Hadamard (entrywise) product of two totally positive Toeplitz matrices remains totally positive, provided that their Toeplitz sequences are the coefficient sequences of series in the Laguerre--Pólya class.

For an entire function \( p(t) = \prod_{i\ge0}(1 + \alpha_it) \) in the Laguerre--Pólya class, if we replace the nonnegative real numbers \( \alpha_i \ge 0 \) by variables \( x_i \), then the coefficient of $t^k$ in \( p(t) \) becomes the \emph{elementary symmetric function} \( e_k(\vec{x}) \) in the variables \( \vec x = (x_1, x_2 ,\dots) \). In this way, the Toeplitz matrix whose entries are given by the coefficients of \( p(t) \) can be regarded as a special case of the following matrix:
\[
   M(\vec{x}) = \bigl(e_{j-i}(\vec{x})\bigr)_{i,j \ge 0}.
\]
The minors of this matrix are determinants of the \emph{(dual) Jacobi--Trudi matrices}\footnote{Our convention differs from the standard one by a conjugation of the skew shape \( \lambda/\mu \); see \Cref{rmk: nonstandard convention}.}
\[
   \tJT_{\lambda/\mu}(\vec{x}) := \bigl(e_{\lambda_i - \mu_j - i + j}(\vec{x})\bigr)_{i,j=1}^{\ell(\lambda)}.
\]
Here \(\lambda/\mu\) is a skew shape, and \( \ell(\lambda) \) denotes the number of parts of the partition \( \lambda \). By the Jacobi--Trudi identity and the Littlewood--Richardson rule, every minor of \( M(\vec{x}) \) is Schur-positive, and hence monomial-positive. Schur positivity and monomial positivity are stronger notions of positivity than positivity of real numbers. 
In connection with Mal\'{o}’s theorem, one may then ask whether the total positivity of the Hadamard product of two Toeplitz matrices \( M(\vec{x}) \) and \( M(\vec{y}) \), when \( x_i, y_j \ge 0 \), can be strengthened to the monomial positivity of each minor. Very recently, Sokal~\cite{Sokal2024} formulated the following conjecture on the monomial positivity of the Hadamard product of \( M(\vec{x}) \) and \( M(\vec{y}) \) for distinct sequences of variables \( \vec{x} = (x_1, x_2, \dots) \) and \( \vec{y} = (y_1, y_2, \dots) \).

\begin{conj}[Sokal's conjecture]\label{conj: Sokal's conjecture}
  The Hadamard product \( M(\vec x) * M(\vec y) \) is totally
  monomial positive. In other words, every minor
\[
    \det(\tJT_{\lambda/\mu}(\vec x) * \tJT_{\lambda/\mu}(\vec y))
\]
is a multi-symmetric function in the variables \( \vec x \) and \( \vec y \) with
nonnegative integer coefficients.
\end{conj}

In the present work, we study a strengthening of Sokal's conjecture along several axes. First, we allow products of Jacobi--Trudi matrices indexed by different skew shapes. In other words, we consider the Hadamard product of an arbitrary minor of $M(\vec{x})$ with an arbitrary minor of $M(\vec{y})$. Second, we consider Hadamard products of $k$ Jacobi--Trudi matrices, rather than just two. Lastly, we generalize the determinant to an arbitrary \textit{Temperley--Lieb immanant}, which is a certain generalization of the determinant introduced by Rhoades and Skandera \cite{Rhoades2005}. 

\begin{conj}\label{conj: stronger sokal's conjecture}
  Given any Temperley--Lieb immanant $\imm_{\tau}$ and any family of
  skew shapes
  $\lambda^{(1)} / \mu^{(1)}, \ldots, \lambda^{(k)} / \mu^{(k)}$, the
  multi-symmetric function
    \[\imm_\tau \left(
    \tJT_{\lambda^{(1)} / \mu^{(1)}} (\vec{x}^{(1)}) 
    * \cdots * 
    \tJT_{\lambda^{(k)} / \mu^{(k)}} (\vec{x}^{(k)}) 
    \right)\]
    is monomial positive.
\end{conj}

While \Cref{conj: Sokal's conjecture} and our stronger \Cref{conj: stronger sokal's conjecture} remain open, we prove that the conjectures hold in the special case when each $\lambda^{(i)} / \mu^{(i)}$ is a skew shape not containing a $3 \times 2$ block of cells. In fact, in this scenario, we demonstrate that the expression in \Cref{conj: stronger sokal's conjecture} is Schur positive.

\begin{thm}\label{thm: TL pos for 3x2 avoiding}
Suppose $\lambda^{(1)} / \mu^{(1)}, \ldots, \lambda^{(k)} / \mu^{(k)}$ is a collection of skew shapes each not containing a $3 \times 2$ block of cells. Then for any Temperley--Lieb immanant $\imm_{\tau}$, the multi-symmetric function
\[\imm_{\tau} \left(\tJT_{\lambda^{(1)} / \mu^{(1)}} (\vec{x}^{(1)}) 
* \cdots * 
\tJT_{\lambda^{(k)} / \mu^{(k)}} (\vec{x}^{(k)}) \right)\]
is Schur positive.
\end{thm}

We focus on this restricted class of skew shapes for three key reasons. First, \Cref{conj: Sokal's conjecture} and \Cref{conj: stronger sokal's conjecture} will not always extend to Schur positivity if the skew shapes do not all avoid a $3 \times 2$ block. For example, the expression $\det(\tJT_{(2,2,2)}(\vec{x}) * \tJT_{(2,2,2)}(\vec{y}))$ is not Schur positive.

Second, the proof of \Cref{thm: TL pos for 3x2 avoiding} involves techniques introduced by Rhoades and Skandera in \cite{Rhoades2005}. Said techniques can be viewed as a generalization of the sign-reversing involution normally used to prove the Lindstr\"{o}m--Gessel--Viennot lemma. 
However, we demonstrate that a similar sign-reversing involution argument cannot work to prove \Cref{conj: Sokal's conjecture} or \Cref{conj: stronger sokal's conjecture} when the collection of skew shapes is not \textit{essentially $(3 \times 2)$-avoiding}, a mild generalization of the condition in \Cref{thm: TL pos for 3x2 avoiding}. See \Cref{prop: counterexample} and the ensuing discussion.

Third, skew shapes avoiding a $3 \times 2$ block generalize the well-known \emph{ribbons}, which are skew shapes avoiding a $2 \times 2$ block. In the special case when each $\lambda^{(i)} / \mu^{(i)}$ is a ribbon, we can write the Schur expansion from \Cref{thm: TL pos for 3x2 avoiding} explicitly in terms of standard Young tableaux with restricted descent sets.
See \Cref{thm: explicit ribbon expansion in full generality}.

In particular, when taking the Hadamard product of a ribbon Jacobi--Trudi matrix with
itself, we obtain the following the following explicit formula.

\begin{thm}\label{thm: Schur expansion of JT_R*JT_R}
Let \( R \) be a ribbon of size \( m \). Then we have a manifestly positive Schur expansion
\[
    \det\left( \tJT_R(\vec x) * \tJT_R(\vec y) \right) = \sum_{\lambda, \mu \vdash m} \sum_{\substack{I, J \subseteq [m-1] \\ I \cap J = \Des(R')}} f^{\lambda}(I) f^{\mu}(J) \, s_\lambda(\vec x) s_\mu(\vec y),
  \]
  where \( s_\lambda \) is the Schur function, \( f^{\lambda}(I) \)
  is the number of standard Young tableaux of shape \( \lambda \)
  with descent set \( I \), and \( \Des(R) \) is the descent set of
  \( R \).
\end{thm}
We explore the latter result further by constructing an
\( \mathfrak{S}_n \times \mathfrak{S}_n \)-module whose Frobenius
image equals the determinant in \Cref{thm: Schur expansion of
  JT_R*JT_R}, thereby establishing Schur positivity from the
representation-theoretic perspective. See \Cref{thm: representation theoretic model for Phi(s_R)}.

The rest of this paper is organized as follows. In Section~\ref{Sec:
  Preliminaries}, we review background material on multi-symmetric
functions, lattice paths, immanants, and the representation theory of products of symmetric groups. 
Section~\ref{Sec: main proofs for 3x2 avoiders} is devoted to the proofs of Theorems~\ref{thm: TL pos for 3x2 avoiding}~and~\ref{thm: Schur expansion of JT_R*JT_R}. 
Section \ref{Sec: rep theory construction} concerns the construction of the \( \mathfrak{S}_n \times \mathfrak{S}_n \)-module that realizes \Cref{thm: Schur expansion of JT_R*JT_R}. 
Finally, Section~\ref{Sec: Concluding remarks} discusses some open problems and
directions for future research.

\section*{Acknowledgments}
The authors are grateful to Yusra Naqvi, Alan Sokal, Richard Stanley, and Lauren Williams for helpful comments.
J. S. Kim was supported by the National Research Foundation of Korea (NRF) grant funded by the Korea government (MSIT) RS-2025-00557835.
J. Oh was supported by NRF grant MSIT NRF-2022R1C1C1010300, RS-2025-16067413 and KIAS Individual Grant (HP083401). 
R. Angarone and D. Soskin thank the American Institute of Mathematics for their support during the Theory and Applications of Total Positivity Workshop in 2023, during which these authors first had the chance to meet and discuss this work.

\section{Preliminaries}\label{Sec: Preliminaries}
In this section, we introduce basic definitions and notations. We
denote \( [a,b] = \{a,a+1,\dots,b\} \) and \( [n] = \{ 1,\dots,n\} \).
We use $\sn$ to denote the group of permutations of $[n]$, i.e., the
symmetric group on $n$ symbols. We use $\CC[\sn]$ to denote the group
algebra for $\sn$, and identify $1 \in \CC[\sn]$ with the identity
permutation.

For $n_1 < n_2$, we regard $\mathfrak{S}_{n_1} \subset \mathfrak{S}_{n_2}$ by considering any permutation of $[n_1]$ as a permutation of $[n_2]$ where every element in $[n_1 + 1, n_2]$ is a fixed point.
We let $\s_i \in \sn$ denote the simple transposition $(i, i+1)$. 
By convention, we compose permutations from left to right, so that, for example, $\s_1\s_2=312$ in one-line notation for \( n=3 \).

\subsection{Multi-symmetric functions}

A \emph{partition} is a sequence
\( \lambda = (\lambda_1, \dots, \lambda_\ell) \) of positive integers
arranged in nonincreasing order. Each \( \lambda_i \) is called a
\emph{part} of \( \lambda \). The \emph{size} \( |\lambda| \) of
$\lambda$ is defined by
\( |\lambda|=\lambda_1 + \cdots + \lambda_\ell \). The \emph{length}
\( \ell(\lambda) \) of \( \lambda \) is defined by
\( \ell(\lambda) = \ell \). We use the standard convention that
\( \lambda_i=0 \) for \( i>\ell(\lambda) \). If a partition
\( \lambda \) has size \( n \), we write \( \lambda \vdash n \).
Denote the set of all partitions by \( \Par \). The \emph{Young
  diagram} of \( \lambda \) is the set
\[
  \{(i,j)\in \ZZ^2: 1\le i\le \ell(\lambda), 1\le j\le \lambda_i\}.
\]
We will identify \( \lambda \) with its Young diagram. Each element
\( (i,j)\in \lambda \) is called a \emph{cell}. We visualize
\( \lambda \) as an array of left-justified square cells, where the
\( i \)-th row contains \( \lambda_i \) cells.
Accordingly, we may also call $\ell(\lambda)$ the \textit{number of rows} in $\lambda$.
For example,
the Young diagram of \( (4,3,1) \) is
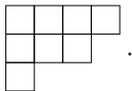
\[
  \vcenter{\hbox{
      \scalebox{0.7}{
  \begin{ytableau}
    ~ & & &  \\
    ~ & & \\
    ~
  \end{ytableau}
}}}.
\]
The \emph{conjugate}
\( \lambda' \) of \( \lambda \) is the partition whose Young diagram
is \( \{(j,i): 1\le i\le \ell(\lambda), 1\le j\le \lambda_i\} \).

For partitions \( \lambda\) and \( \mu \), we write
\( \mu \subseteq \lambda \) if 
$\mu_i \leq \lambda_i$ for all $i$.
In this case, the \emph{skew
  shape} \( \lambda / \mu \) is the set-theoretic difference
\( \lambda-\mu \) of their Young diagrams, and \( |\lambda/\mu| \) is
the number of cells in \( \lambda / \mu \). 
The \emph{length} of a skew shape 
$\lambda / \mu$ is defined to be $\ell(\lambda)$.
The \emph{conjugate}
\( (\lambda/\mu)' \) of \( \lambda/\mu \) is defined to be
\( \lambda'/\mu' \).

A \emph{semistandard Young tableau} of shape \(\lambda/\mu\) is a
filling of \(\lambda/\mu\) with positive integers such that the
entries weakly increase from left to right in each row and strictly
increase from top to bottom in each column. We denote by
\( \SSYT(\lambda/\mu) \) the set of semistandard Young tableaux of
shape \( \lambda/\mu \). A \emph{standard Young tableau} of shape
\( \lambda/\mu \) is a semistandard Young tableau of shape
\( \lambda/\mu \) with a bijective filling, i.e., each of
\(1,2,\dots,|\lambda/\mu|\) appears exactly once. We denote by
\( \SYT(\lambda/\mu) \) the set of standard Young tableaux of shape
\( \lambda/\mu \). The \emph{descent set} \(\Des(T)\) of a standard
Young tableau \(T\) is defined as the set of indices \(i\) such that
the entry \(i+1\) appears in a row strictly below than that of \(i\).

Let \( \vec x= (x_1, x_2, \dots) \) be an infinite sequence of
variables. A \emph{symmetric function} in the variables \( \vec x \)
is a formal power series that is symmetric in the sense that for any
permutation \( \sigma \) of the variables,
\[
f(x_{\sigma(1)}, x_{\sigma(2)}, \dots) = f(x_1, x_2, \dots).
\]
Let \( \Lambda(\vec x) \) denote the graded ring of symmetric
functions in \( \vec x \) over \( \mathbb{Q} \). For a partition
\( \lambda = (\lambda_1,\dots,\lambda_\ell) \), the \emph{monomial
  symmetric function} \( m_\lambda \) is defined as
\[
    m_\lambda(\vec x) := \sum_{\alpha} ( x_1^{\alpha_1} x_2^{\alpha_2} \cdots),  
\]
where the summation is over all permutations
\( \alpha = (\alpha_1,\alpha_2,\dots) \) of
\( (\lambda_1,\dots,\lambda_\ell, 0, 0,\dots) \). The
\emph{homogeneous} and \emph{elementary} symmetric functions
associated with a partition \( \lambda \) are defined by
\[
    h_\lambda(\vec x) := h_{\lambda_1}(\vec x) h_{\lambda_2}(\vec x) \cdots \quad \text{and} \quad e_\lambda(\vec x) := e_{\lambda_1}(\vec x) e_{\lambda_2}(\vec x) \cdots,  
\]  
respectively, where
\[
    h_d(\vec x) := \sum_{i_1 \leq \cdots \leq i_d} x_{i_1} \cdots x_{i_d} \quad \text{and} \quad e_d(\vec x) := \sum_{i_1 < \cdots < i_d} x_{i_1} \cdots x_{i_d}.  
\]
We define \( h_0(\vec x) = e_0(\vec x)=1 \) and \( h_d(\vec x) = e_d(\vec x)=0 \) for \( d<0 \).
We will sometimes omit the variables \( \vec x \) and simply write
\( \Lambda \) and \( m_\lambda \), etc., when the choice of
\( \vec x \) is irrelevant.

One of the most prominent symmetric functions is the Schur function \( s_\lambda \). The \emph{(skew) Schur function} \( s_{\lambda/\mu} \) is defined via the (dual) Jacobi--Trudi identity:
\begin{equation}\label{eq: JT identity}
  s_{\lambda / \mu}(\vec x) = \left( e_{\lambda'_i - \mu'_j  - i+ j}(\vec x) \right)_{i,j=1}^{\ell(\lambda')}.
\end{equation} 
When \( \mu = \emptyset \), the function
\( s_\lambda(\vec x) := s_{\lambda/\emptyset}(\vec x) \) is referred
to as the \emph{Schur function} indexed by \( \lambda \).

We define \emph{(dual) Jacobi--Trudi matrices} \( \tJT_{\lambda / \mu} (\vec{x}) \) by
\[\tJT_{\lambda / \mu} (\vec{x}) :=  \left( e_{\lambda_i - \mu_j  - i+ j}(\vec x) \right)_{i,j=1}^{\ell(\lambda)}.\]
Note that, with this notation,
$\det \tJT_{\lambda / \mu} (\vec x) = s_{\lambda' / \mu'}(\vec x)$.

\begin{remark}\label{rmk: nonstandard convention}
Our definition of dual Jacobi--Trudi matrices
differs from the standard convention by a conjugation of $\lambda / \mu$.
When using the standard convention, the reader should take this change into account.
For example, results stated in terms of $(3 \times 2)$-avoiding skew shapes
would instead concern $(2 \times 3)$-avoiding skew shapes.
\end{remark}

Note that each \( \det \left(\tJT_{\lambda/\mu}(\vec x)\right) \) is a minor of the matrix
\( M(\vec x) = (e_{i-j}(\vec x))_{i,j\ge0} \) introduced earlier, 
and in fact all minors of $M(\vec x)$ arise this way.

Schur functions admit a combinatorial expression as a generating
function for semistandard Young tableaux:
\begin{equation}\label{eq: tableaux description}
    s_{\lambda/\mu}(\vec x) = \sum_{T \in \SSYT(\lambda/\mu)} \vec x_T,
\end{equation}
where \( \vec x_T = x_1^{a_1}x_2^{a_2}\cdots \), with \( a_i \) denoting the number of entries equal to \( i \) in \( T \).

Each of the families 
\[
\{ m_\lambda : \lambda \in \Par \}, \quad
\{ h_\lambda : \lambda \in \Par \}, \quad
\{ e_\lambda : \lambda \in \Par \}, \quad
\{ s_\lambda : \lambda \in \Par \}
\]
forms a basis for the space of symmetric functions \( \Lambda \). We
denote by \( \omega \) the standard involution on \( \Lambda \), which
sends \( h_\lambda \mapsto e_\lambda \).

In this paper, we consider the tensor product
\( \Lambda(\vec x^{(1)}) \otimes \cdots \otimes \Lambda(\vec x^{(k)})
\) of symmetric function rings for \( k\ge1 \). An element
\( f \in \Lambda(\vec x^{(1)}) \otimes \cdots \otimes \Lambda(\vec
x^{(k)}) \) is called a \emph{multi-symmetric function}. For
\( f^{(1)} \in \Lambda(\vec x^{(1)}),\dots, f^{(k)} \in \Lambda(\vec
x^{(k)}) \), we abuse notation by writing the multi-symmetric function
\( f^{(1)} \otimes \cdots \otimes f^{(k)} \) as
\( f^{(1)}(\vec x^{(1)})\cdots f^{(k)}(\vec x^{(k)}) \).
We extend the definition of the involution
\( \omega \) to
\[
  \omega: \Lambda(\vec x^{(1)}) \otimes \cdots \otimes \Lambda(\vec x^{(k)}) \to \Lambda(\vec x^{(1)}) \otimes \cdots \otimes \Lambda(\vec x^{(k)})
\]
by applying \(\omega\) to each tensor factor, that is,
\[
  \omega\left(h_{\lambda^{(1)}}(\vec x^{(1)}) \cdots h_{\lambda^{(k)}}(\vec x^{(k)}) \right) = e_{\lambda^{(1)}}(\vec x^{(1)}) \cdots e_{\lambda^{(k)}}(\vec x^{(k)}).
\]

Given a basis
\( \{ b_\lambda \} \) of \( \Lambda \), we say that a
multi-symmetric function
\( f(\vec x^{(1)}, \dots,\vec x^{(k)}) \in \Lambda(\vec x^{(1)})
\otimes \cdots \otimes \Lambda(\vec x^{(k)})\) is
\emph{\(b\)-positive} if it can be expressed positively in terms of
the tensor product basis, that is, if
\[
  f(\vec x^{(1)}, \dots,\vec x^{(k)})
  = \sum_{\lambda^{(1)}, \dots, \lambda^{(k)} \in \Par} c_{\lambda^{(1)},\dots,\lambda^{(k)}}
  b_{\lambda^{(1)}}(\vec x^{(1)}) \cdots b_{\lambda^{(k)}}(\vec x^{(k)}),
\]
then \( c_{\lambda^{(1)},\dots,\lambda^{(k)}} \geq 0 \) for all
\( \lambda^{(1)}, \dots, \lambda^{(k)} \in \Par \). 

Suppose we are given a list of matrices 
    \[ A^{(1)} = \left(a^{(1)}_{i,j}(\vec{x}^{(1)})\right)_{1\le i,j\le n} 
    \, , \,\,\, \dots \,\,\, , \,
    A^{(k)} = \left(a^{(k)}_{i,j}(\vec{x}^{(k)})\right)_{1\le i,j\le n}, \]
where the entries
$a^{(p)}_{i,j}(\vec{x}^{(p)})$ of $A^{(p)}$
are symmetric functions in \( \vec{x}^{(p)} \) for all $1 \leq p \leq k$.
We define the \emph{Hadamard product}
\( A^{(1)} * \cdots * A^{(k)} \) as the entrywise product:
\[
A^{(1)} * \cdots * A^{(k)} 
= \left(
a^{(1)}_{i,j}(\vec{x}^{(1)}) 
\cdots
a^{(k)}_{i,j}(\vec{x}^{(k)})
\right)_{1\le i,j\le n},
\]
whose entries lie in \( \Lambda(\vec{x}^{(1)}) \otimes \cdots \otimes \Lambda(\vec{x}^{(k)}) \).

In \Cref{Sec: main proofs for 3x2 avoiders}, we will consider Hadamard products of several Jacobi--Trudi matrices, possibly for skew shapes with different numbers of rows.
To that end, we define
\[\tJT_{\lambda / \mu}^n (\vec{x}) :=  \left( e_{\lambda_i - \mu_j  - i+ j}(\vec x) \right)_{i,j=1}^n.\]
Note that, for any \( m,n\ge \ell(\lambda) \), we have
\[
  \det  \left( \tJT_{\lambda / \mu}^n (\vec{x})  \right)  = \det  \left( \tJT_{\lambda / \mu}^m (\vec{x})  \right) ,
\]
which follows easily from the definition.
Hence, we write
\( \det \left( \tJT_{\lambda / \mu} (\vec{x}) \right) \)
to mean 
\( \det ( \tJT_{\lambda / \mu}^n (\vec{x}) ) \) for some \( n\ge \ell(\lambda) \).
More generally,
whenever \( m,n \ge \ell(\lambda^{(i)}) \) for all \( i \), we have
\[
\det \left(  
\tJT_{\lambda^{(1)} / \mu^{(1)}}^n (\vec{x}^{(1)}) 
* \cdots * 
\tJT_{\lambda^{(k)} / \mu^{(k)}}^n (\vec{x}^{(k)})  \right)
= \det \left(
\tJT_{\lambda^{(1)} / \mu^{(1)}}^m (\vec{x}^{(1)}) 
* \cdots * 
\tJT_{\lambda^{(k)} / \mu^{(k)}}^m (\vec{x}^{(k)}) \right).
\]
So, we define any of these expressions to be
\[
\det
\left(  \tJT_{\lambda^{(1)} / \mu^{(1)}} (\vec{x}^{(1)}) 
* \cdots * 
\tJT_{\lambda^{(k)} / \mu^{(k)}} (\vec{x}^{(k)})  \right)
\]
without risk of confusion. The same is true when the determinant is replaced with any immanant $\imm_f$; see \Cref{sec: imms and tl setup}.

\subsection{Lattice paths and networks}\label{sec:lattice paths and networks}

The key idea for the proofs of \Cref{thm: TL pos for 3x2 avoiding} and \Cref{thm: Schur expansion of JT_R*JT_R} is to interpret the entries of Jacobi--Trudi matrices as generating functions for paths in directed graphs, in the spirit of the Lindstr\"om--Gessel--Viennot lemma.

Let $\mathcal{L}$ be the directed graph with vertex set $\mathbb{Z}^2$
and edges of the form $(a,b) \to (a,b-1)$ and $(a,b) \to (a-1,b-1)$ for all $(a,b) \in \mathbb{Z}^2$.
We call these edges \textit{south} and \textit{southwest} steps, respectively.
We will refer to paths in $\mathcal{L}$ as \emph{(south-southwest) lattice paths}.
For $A,B \in \mathbb{Z}^2$, we write $p:A \to B$ to mean that $p$ is a south-southwest lattice path
that begins at $A$ and ends at $B$.
Suppose $p$ is a south-southwest lattice path
and $\vec{x} = (x_1, x_2, \ldots)$ is an infinite sequence of variables.
Then let $\wt_{\vec x}(p)$ denote the \textit{$\vec x$-weight} of $p$,
which is the product of $x_{b}$ for each southwest step $(a,b) \to (a-1,b-1)$ in $p$.
See \Cref{fig:path} for an example.

\begin{figure}
  \centering
  \begin{tikzpicture}[scale=0.5]
    \draw[help lines] (0,0) grid (7,5);
    \foreach \x in {0,...,7}
    \draw (\x,0) node[below] {\x};
    \foreach \y in {0,...,5}
    \draw (0,\y) node[left] {\y};
    \draw[line width = 1pt] (3,0) --++(0,1) --++(1,1)--++(0,1) --++(1,1)--++(1,1);
  \end{tikzpicture}
  \caption{A lattice path \( p:A \to B \) from \( A=(6,5) \) to \( B=(3,0) \)
  with \( \wt_{\vec x}(p) = x_5x_4x_2 \).}
  \label{fig:path}
\end{figure}

To model Jacobi--Trudi matrices, we allow certain infinite lattice paths.
We write $p:(a,\infty) \to B$ to mean that,
for some $N \in \mathbb{Z}$,
the path $p$ contains infinitely many south steps of the form $(a,M) \to (a,M-1)$ for $M > N$
followed by a finite south-southwest lattice path $p:(a,N) \to B$.

Suppose $\lambda / \mu$ is a skew shape with \( \ell(\lambda)\le n \).
Then, for each $1 \leq i \leq n$, define
    \[\begin{array}{ccc}
    A(\lambda)_i := (\lambda_i + n-i, \infty) & \text{and} & B(\mu)_i := (\mu_i + n-i, 0).
    \end{array}\]
Observe that
\[\sum_{p:A(\lambda)_i \to B(\mu)_j} \wt_{\vec x}(p) = e_{\lambda_i-\mu_j-i+j}(\vec x) = (\tJT^n_{\lambda/\mu}(\vec x))_{i,j}.\]
A \emph{$(\lambda, \mu)$-path family} is a tuple $\vec{p} = (p_1, \ldots, p_n)$ of lattice paths with $p_i:A(\lambda)_i \to B(\mu)_{\sigma(i)}$ for all $1 \leq i \leq n$ and some permutation $\sigma \in \mathfrak{S}_n$.
The permutation $\sigma$ is called the \emph{type} of $\vec{p}$ and is denoted $\type(\vec{p})$. The \emph{$\vec x$-weight} of a path family is the product of the $\vec x$-weights of the individual paths in $\vec{p}$, and is denoted $\wt_{\vec x}(\vec{p})$.
Two paths are said to be \emph{nonintersecting} if they have no vertices in common.
A path family is said to be \emph{nonintersecting} if its paths are pairwise nonintersecting.

With this notation in hand, we may express the determinant of a dual Jacobi--Trudi matrix as
    \begin{equation}
    \det \left( \tJT^n_{\lambda/\mu}(\vec x) \right) = \sum_{\sigma \in \mathfrak{S}_n} \sum_{ \type(\vec{p}) = \sigma} \sgn(\sigma) \wt_{\vec x}(\vec{p}), \label{eq: first det expansion}
    \end{equation}
where the inner sum ranges over all $(\lambda, \mu)$-path families $\vec{p}$.
Since the graph \( \mathcal{L} \) is planar, by the Lindstr\"om--Gessel--Viennot lemma, we have the following identity
in the case of dual Jacobi--Trudi matrices:
\begin{equation}\label{eq:LGV}
  \det \left( \tJT^n_{\lambda/\mu}(\vec x) \right) = \sum_{\vec{p}} \wt_{\vec x}(\vec{p}),
\end{equation}
where the sum ranges over all nonintersecting $(\lambda, \mu)$-path families $\vec{p}$.

As stated in the previous subsection, $\det ( \tJT^n_{\lambda/\mu}(\vec x) )$
does not depend on the choice of $n \geq \ell(\lambda)$. 
This can also be seen from \eqref{eq:LGV}: increasing the value of $n$ corresponds to appending vertical paths with trivial weight to each path family $\vec p$.
Accordingly, in what follows, we sometimes omit the parameter $n$ when discussing $(\lambda, \mu)$-path families.

\subsection{Temperley--Lieb Immanants}\label{sec: imms and tl setup}

Given a function $f: \mathfrak{S}_n \rightarrow \CC$ and an $n \times n$ matrix $M$, the \emph{$f$-immanant} of $M$ is
\begin{equation}\label{eq:immdef}
\imm_{f}(M) := \sum_{w \in \sn} f(w) M_{1, w(1)} \cdots M_{1, w(n)}.
\end{equation}
One particularly well-studied family of immanants are \emph{ordinary immanants}. 
Ordinary immanants are obtained by setting \( f=  \chi^\lambda \), which is the
irreducible character of \( \mathfrak{S}_n \) associated to the partition $\lambda \vdash n$. 
Notably, the ordinary immanants with \( f=\chi^{(1^n)} \) and \( f=\chi^{(n)} \) coincide with
the determinant and the permanent, respectively.

The study of immanants has significantly expanded the theory of
determinants. They have been investigated in contexts such as
semi-definite Hermitian matrices \cite{Schur1918}, totally positive
matrices \cite{Stembridge1991}, and Jacobi--Trudi matrices. Goulden
and Jackson \cite{Goulden1992a} initiated the study of immanants of
Jacobi--Trudi matrices. They conjectured that ordinary immanants of
Jacobi--Trudi matrices are monomial positive. This conjecture was later
confirmed by Greene \cite{Greene1992}.

In addition to ordinary immanants, \emph{Temperley--Lieb immanants} have been studied for Jacobi--Trudi matrices. These immanants were introduced by Rhoades and Skandera \cite{Rhoades2005} and can be defined via the combinatorics of nonintersecting matchings. Much like ordinary immanants, they include the determinant as a special case. 
In \cite{RHOADES2006793}, Rhoades and Skandera showed the Schur positivity of Temperley--Lieb immanants of generalized Jacobi--Trudi matrices. 
They did this by demonstrating that Temperley--Lieb immanants are a subclass of \textit{Kazhdan--Lusztig immanants} \cite[Proposition 5]{RHOADES2006793}, 
whose Schur positivity follows from the work of Haiman \cite[Theorem 1.5]{Haiman1993}.
In their recent work \cite{Nguyen2025}, Nguyen and Pylyavskyy used crystal operators to provide a combinatorial interpretation for the coefficients of Temperley--Lieb immanants of Jacobi--Trudi matrices when expanded in the Schur basis.

We now present the basic definitions underlying Temperley--Lieb immanants. 
Given a formal parameter $\xi$, the \emph{Temperley--Lieb algebra} $\TL_n(\xi)$ is the associative $\CC$-algebra generated by
$t_1,\dotsc,t_{n-1}$ and subject to the relations
\begin{alignat*}{2}
t_i^2 &= \xi t_i, &\qquad &\text{for } i=1,\dotsc,n-1, \\
t_i t_j t_i &= t_i,   &\qquad &\text{if }  |i-j|=1,\\
t_i t_j &= t_j t_i,   &\qquad &\text{if }  |i-j| \geq 2.
\end{alignat*}

In the remainder of the paper, we assume \( \xi=2 \). In this case, there is an isomorphism \cite{FanGreen} \[\TL_n(2) \cong \csn/(1 + \s_1 + \s_2 + \s_1\s_2 + \s_2\s_1 + \s_1\s_2\s_1).\]
Specifically, the isomorphism is given by the algebra homomorphism
    \begin{equation}\label{eq:sntotn}
    \begin{aligned}
    \theta : \csn \rightarrow \TL_n(2) ~\mbox{ with }~ 
    \theta(\s_i) = t_i - 1.
    \end{aligned}
    \end{equation}
    
One may interpret multiplication in $\TL_n(2)$ visually. 
To each algebra generator, we associate a \emph{Kauffman diagram}, which is an (undirected) graph with $2n$ vertices arranged in an $n \times 2$ grid and $n$ nonintersecting edges, such that each edge lies in the convex hull of the $2n$ vertices.
To the generator $t_i$, we associate the Kauffman diagram in which the $i$th and $(i+1)$st points from the bottom are connected by an arc on either side, and all other arcs go horizontally from left to right. The multiplicative identity is associated with the diagram in which all arcs go horizontally from left to right. Examples with $n=4$ are given in \Cref{fig: TL gens example}.

\begin{figure}
\centering
\[
\begin{tikzpicture}[x=0.2cm,y=0.5cm,thick]
  \foreach \i/\y in {1/0, 2/1.2, 3/2.4, 4/3.6} {
    \draw (0,\y) -- (4,\y);
    \fill (0,\y) circle (2pt) node[left] {\scriptsize $A_{\i}$};
    \fill (4,\y) circle (2pt) node[right] {\scriptsize $B_{\i}$};
  }
  \node at (2,-0.8) {\scriptsize $1$};

  \begin{scope}[xshift=3.8cm]
    \foreach \i/\y in {1/0, 2/1.2} {
      \fill (0,\y) circle (2pt) node[left] { \scriptsize $A_{\i}$};
      \fill (4,\y) circle (2pt) node[right] { \scriptsize $B_{\i}$};
    }
    \foreach \i/\y in {3/2.4, 4/3.6} {
      \draw (0,\y) -- (4,\y);
      \fill (0,\y) circle (2pt) node[left] { \scriptsize $A_{\i}$};
      \fill (4,\y) circle (2pt) node[right] { \scriptsize $B_{\i}$};
    }
    \draw (0,0) .. controls (0.8,0.6) .. (0,1.2);
    \draw (4,0) .. controls (3.2,0.6) .. (4,1.2);
    \node at (2,-0.8) {\scriptsize $t_1$};
  \end{scope}

  \begin{scope}[xshift=7.6cm]
    \draw (0,0) -- (4,0);
    \fill (0,0) circle (2pt) node[left] { \scriptsize $A_1$};
    \fill (4,0) circle (2pt) node[right] { \scriptsize $B_1$};
    \fill (0,1.2) circle (2pt) node[left] { \scriptsize $A_2$};
    \fill (4,1.2) circle (2pt) node[right] { \scriptsize $B_2$};
    \fill (0,2.4) circle (2pt) node[left] { \scriptsize $A_3$};
    \fill (4,2.4) circle (2pt) node[right] { \scriptsize $B_3$};
    \draw (0,3.6) -- (4,3.6);
    \fill (0,3.6) circle (2pt) node[left] { \scriptsize $A_4$};
    \fill (4,3.6) circle (2pt) node[right] { \scriptsize $B_4$};
    \draw (0,1.2) .. controls (0.8,1.8) .. (0,2.4);
    \draw (4,1.2) .. controls (3.2,1.8) .. (4,2.4);
    \node at (2,-0.8) {\scriptsize $t_2$};
  \end{scope}

  \begin{scope}[xshift=11.4cm]
    \draw (0,0) -- (4,0);
    \draw (0,1.2) -- (4,1.2);
    \foreach \i/\y in {1/0, 2/1.2, 3/2.4, 4/3.6} {
      \fill (0,\y) circle (2pt) node[left] { \scriptsize $A_{\i}$};
      \fill (4,\y) circle (2pt) node[right] { \scriptsize $B_{\i}$};
    }
    \draw (0,2.4) .. controls (0.8,3.0) .. (0,3.6);
    \draw (4,2.4) .. controls (3.2,3.0) .. (4,3.6);
    \node at (2,-0.8) {\scriptsize $t_3$};
  \end{scope}
\end{tikzpicture}
\]
\caption{Multiplicative generators for $\TL_4(2)$.}
\label{fig: TL gens example}
\end{figure}
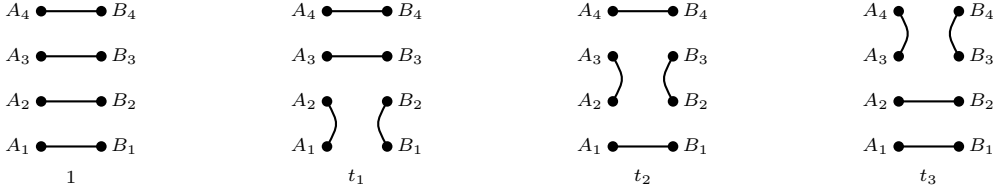

To multiply Kauffman diagrams, we concatenate them by identifying the right $n$ points of one diagram with the left $n$ points of the next. We remove each loop formed in this process and multiply the result by $2$ for each loop removed. An example with $n=4$ is given in \Cref{fig: example of multiplication in TL}. As is apparent from the visual interpretation of multiplication, the algebra $\TL_n(2)$ is finite-dimensional, with a basis $\K_n$ labeled by all possible Kauffman diagrams. It is known that $|\K_n|$ is the $n$th Catalan number $C_n = \frac{1}{n+1}\tbinom{2n}{n}$.

\begin{figure}
\centering
\begin{tikzpicture}[x=0.2cm,y=0.5cm,thick]

\begin{scope}[xshift=0cm]
  \foreach \y in {0,1.2} { \fill (0,\y) circle (2pt); \fill (4,\y) circle (2pt); }
  \foreach \y in {2.4,3.6} { \draw (0,\y) -- (4,\y); \fill (0,\y) circle (2pt); \fill (4,\y) circle (2pt); }
  \draw (0,0) .. controls (0.8,0.6) .. (0,1.2);
  \draw (4,0) .. controls (3.2,0.6) .. (4,1.2);
  \node at (2,-0.8) {\scriptsize $t_1$};
\end{scope}

\begin{scope}[xshift=1.2cm]
  \draw (0,0) -- (4,0);
  \foreach \y in {0,1.2,2.4,3.6} { \fill (0,\y) circle (2pt); \fill (4,\y) circle (2pt); }
  \draw (0,1.2) .. controls (0.8,1.8) .. (0,2.4);
  \draw (4,1.2) .. controls (3.2,1.8) .. (4,2.4);
  \draw (0,3.6) -- (4,3.6);
  \node at (2,-0.8) {\scriptsize $t_2$};
\end{scope}

\begin{scope}[xshift=2.4cm]
  \foreach \y in {0,1.2} { \fill (0,\y) circle (2pt); \fill (4,\y) circle (2pt); }
  \foreach \y in {2.4,3.6} { \draw (0,\y) -- (4,\y); \fill (0,\y) circle (2pt); \fill (4,\y) circle (2pt); }
  \draw (0,0) .. controls (0.8,0.6) .. (0,1.2);
  \draw (4,0) .. controls (3.2,0.6) .. (4,1.2);
  \node at (2,-0.8) {\scriptsize $t_1$};
\end{scope}

\begin{scope}[xshift=3.6cm]
  \foreach \y in {0,1.2} { \fill (0,\y) circle (2pt); \fill (4,\y) circle (2pt); }
  \foreach \y in {2.4,3.6} { \draw (0,\y) -- (4,\y); \fill (0,\y) circle (2pt); \fill (4,\y) circle (2pt); }
  \draw (0,0) .. controls (0.8,0.6) .. (0,1.2);
  \draw (4,0) .. controls (3.2,0.6) .. (4,1.2);
  \node at (2,-0.8) {\scriptsize $t_1$};
\end{scope}

\begin{scope}[xshift=4.8cm]
  \draw (0,0) -- (4,0);
  \draw (0,1.2) -- (4,1.2);
  \foreach \y in {0,1.2,2.4,3.6} { \fill (0,\y) circle (2pt); \fill (4,\y) circle (2pt); }
  \draw (0,2.4) .. controls (0.8,3.0) .. (0,3.6);
  \draw (4,2.4) .. controls (3.2,3.0) .. (4,3.6);
  \node at (2,-0.8) {\scriptsize $t_3$};
\end{scope}

\node at (31,1.8) {\large $=$};
\node at (60,1.8) {\large $= 2$};

\begin{scope}[xshift=7cm]
  \begin{scope}[xshift=0cm]
    \foreach \y in {0,1.2} { \fill (0,\y) circle (2pt); \fill (4,\y) circle (2pt); }
    \foreach \y in {2.4,3.6} { \draw (0,\y) -- (4,\y); \fill (0,\y) circle (2pt); \fill (4,\y) circle (2pt); }
    \draw (0,0) .. controls (0.8,0.6) .. (0,1.2);
    \draw (4,0) .. controls (3.2,0.6) .. (4,1.2);
    \node at (2,-0.8) {\scriptsize $t_1$};
  \end{scope}

  \begin{scope}[xshift=0.8cm]
    \draw (0,0) -- (4,0);
    \foreach \y in {0,1.2,2.4,3.6} { \fill (0,\y) circle (2pt); \fill (4,\y) circle (2pt); }
    \draw (0,1.2) .. controls (0.8,1.8) .. (0,2.4);
    \draw (4,1.2) .. controls (3.2,1.8) .. (4,2.4);
    \draw (0,3.6) -- (4,3.6);
    \node at (2,-0.8) {\scriptsize $t_2$};
  \end{scope}

  \begin{scope}[xshift=1.6cm]
    \foreach \y in {0,1.2} { \fill (0,\y) circle (2pt); \fill (4,\y) circle (2pt); }
    \foreach \y in {2.4,3.6} { \draw (0,\y) -- (4,\y); \fill (0,\y) circle (2pt); \fill (4,\y) circle (2pt); }
    \draw (0,0) .. controls (0.8,0.6) .. (0,1.2);
    \draw (4,0) .. controls (3.2,0.6) .. (4,1.2);
    \node at (2,-0.8) {\scriptsize $t_1$};
  \end{scope}

  \begin{scope}[xshift=2.4cm]
    \foreach \y in {0,1.2} { \fill (0,\y) circle (2pt); \fill (4,\y) circle (2pt); }
    \foreach \y in {2.4,3.6} { \draw (0,\y) -- (4,\y); \fill (0,\y) circle (2pt); \fill (4,\y) circle (2pt); }
    \draw (0,0) .. controls (0.8,0.6) .. (0,1.2);
    \draw (4,0) .. controls (3.2,0.6) .. (4,1.2);
    \node at (2,-0.8) {\scriptsize $t_1$};
  \end{scope}

  \begin{scope}[xshift=3.2cm]
    \draw (0,0) -- (4,0);
    \draw (0,1.2) -- (4,1.2);
    \foreach \y in {0,1.2,2.4,3.6} { \fill (0,\y) circle (2pt); \fill (4,\y) circle (2pt); }
    \draw (0,2.4) .. controls (0.8,3.0) .. (0,3.6);
    \draw (4,2.4) .. controls (3.2,3.0) .. (4,3.6);
    \node at (2,-0.8) {\scriptsize $t_3$};
  \end{scope}
\end{scope}

\begin{scope}[xshift=12.7cm]
  \foreach \y in {0,1.2} { \fill (0,\y) circle (2pt); \fill (4,\y) circle (2pt); }
  \foreach \y in {2.4,3.6} { \fill (0,\y) circle (2pt); \fill (4,\y) circle (2pt); }
  \draw (0,0) .. controls (0.8,0.6) .. (0,1.2);
  \draw (4,0) .. controls (3.2,0.6) .. (4,1.2);
  \draw (0,2.4) .. controls (0.8,3.0) .. (0,3.6);
  \draw (4,2.4) .. controls (3.2,3.0) .. (4,3.6);
  \node at (2,-0.8) {\scriptsize $t_1t_3$};
\end{scope}
\end{tikzpicture}
\caption{Concatenation of Kauffman diagrams for the product $t_1t_2t_1t_1t_3=2t_1t_3$.}
\label{fig: example of multiplication in TL}
\end{figure}
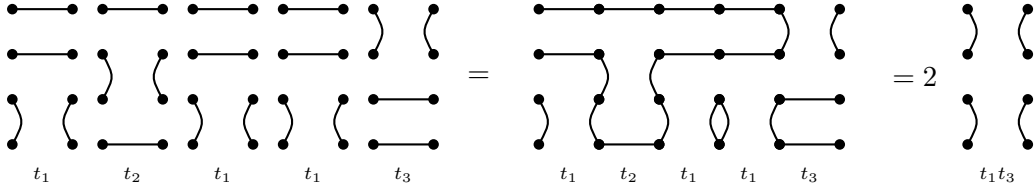

For each $\tau \in \K_n$, the function $f_\tau:\CC[\sn] \to \CC$ is defined via
\[f_\tau(w) = \text{ the coefficient of $\tau$ in } \theta(w)\]
and extended linearly.
Given $\tau \in \K_n$, the \emph{Temperley--Lieb immanant} for $\tau$ is
the \( f \)-immanent \( \imm_f \), where $f=f_\tau$. To economize notation, we will typically write $\imm_{\tau}$ instead of $\imm_{f_\tau}$. Note that if $\tau$ is the diagram associated to the identity element in $\TL_n(2)$, then $f_\tau(w) = \sgn(w)$, so $\imm_\tau$ is the determinant.

\begin{exam}
Let $n=3$ and $\tau=t_1t_2$. Then $$\theta(1)=1, \quad \theta(\s_1)=t_1-1, \quad \theta(\s_2)=t_2-1.$$
It follows that
\begin{align*}
    \theta(\s_1\s_2)&=(t_1-1)(t_2-1)=t_1t_2-t_1-t_2+1, \\
    \theta(\s_2\s_1)&=(t_2-1)(t_1-1)=t_2t_1-t_2-t_1+1, \\
    \theta(\s_1\s_2\s_1)&=(t_1-1)(t_2-1)(t_1-1)=t_1t_2t_1-t_2t_1-t^{2}_1-t_1t_2+t_1+t_2+t_1-1.
\end{align*}
Using relations in $\TL_3(2)$, we simplify the last expression to
$$\theta(\s_1\s_2\s_1)=t_1+t_2-t_1t_2-t_2t_1-1.$$
Finally, we extract the coefficient of $t_1t_2$ in each of the expressions above to obtain 
$$\imm_{t_1t_2}(M)=M_{13}M_{21}M_{32}-M_{13}M_{22}M_{31}.$$
\end{exam}

\subsection{TL immanants of Jacobi--Trudi matrices} \label{sec:imms and path matrices}

Generalizing \eqref{eq: first det expansion}, the $f$-immanant of a dual Jacobi--Trudi matrix can be written as
    \begin{equation}
    \imm_f \left( \tJT_{\lambda/\mu}(\vec x) \right) = \sum_{\sigma \in \mathfrak{S}_n} \sum_{ \type(\vec{p}) = \sigma} f(\sigma) \wt_{\vec x}(\vec{p}),
    \label{eq: first imm expansion}
    \end{equation}
where the inner sum ranges over all $(\lambda,\mu)$-path families $\vec p$.
Rhoades and Skandera \cite{Rhoades2005} proved that Temperley--Lieb immanants of
path matrices for planar networks can be written as nonnegative expressions 
in the associated network's edge weights.
In the remainder of this subsection, we paraphrase the essential steps of their proof,
since some of the vocabulary developed by Rhoades and Skandera
will be useful in the proofs of our main theorems.
For simplicity, we phrase the results in terms of
the south-southwest lattice path graph $\mathcal{L}$
and dual Jacobi--Trudi matrices,
although the setup generalizes to all planar networks.

For a skew shape $\lambda/\mu$,
a \emph{$(\lambda, \mu)$-subnetwork} is a multiset $H$ of edges in $\mathcal{L}$
that is equal to the disjoint union of all paths in some $(\lambda, \mu)$-path family $\vec p$.
In this case, we say that $\vec{p}$ \emph{covers} $H$.
We may also simply call $H$ a \textit{subnetwork} to mean it is a $(\lambda, \mu)$-subnetwork for some $\lambda/\mu$.
While the path family covering a given subnetwork $H$ is not always unique,
the $\vec x$-weight of such a path family is unique.
As such, we write $\wt_{\vec x}(H)$ to denote
the $\vec x$-weight of any path family that covers $H$.
Given a $(\lambda, \mu)$-subnetwork $H$ with $\ell(\lambda) \leq n$,
the \emph{path family generating function} for $H$ is
\[\beta(H) := \sum_{\vec{p}} \type(\vec{p}) \in \CC[\mathfrak{S}_n],\]
where the sum is over all path families \( \vec{p} \) that cover $H$.
See \Cref{fig:subnetwork and beta} for an example.

\begin{figure}
    \centering
    \[\begin{array}{cc}
    \begin{tikzpicture}[scale=.75,baseline={(0,1)}]
    \node (H) at (-.75,2) {$H=$};
    \node[above] at (5,4) {$A_1$};
    \node[above] at (4,4) {$A_2$};
    \node[above] at (3,4) {$A_3$};
    \node[above] at (2,4) {$A_4$};
    \node[below] at (3,0) {$B_1$};
    \node[below] at (2,0) {$B_2$};
    \node[below] at (1,0) {$B_3$};
    \node[below] at (0,0) {$B_4$};
    \draw[dashed,gray] (0,0) grid (5,4);
    \draw[ultra thick,black] (3, 0) -- (3, 1);
    \draw[ultra thick,black] (3, 1) -- (4, 2);
    \draw[ultra thick,black] (4, 2) -- (4, 3);
    \draw[ultra thick,black] (4, 3) -- (5, 4);
    \draw[ultra thick,black] (2, 0) -- (2, 1);
    \draw[ultra thick,black] (2, 1) -- (3, 2);
    \draw[ultra thick,black] (3, 2) -- (4, 3);
    \draw[ultra thick,black] (4, 3) -- (4, 4);
    \draw[ultra thick,black] (1, 0) -- (2, 1);
    \draw[ultra thick,black] (2, 1) -- (2, 2);
    \draw[ultra thick,black] (2, 2) -- (3, 3);
    \draw[ultra thick,black] (3, 3) -- (3, 4);
    \draw[ultra thick,black] (0, 0) -- (1, 1);
    \draw[ultra thick,black] (1, 1) -- (1, 2);
    \draw[ultra thick,black] (1, 2) -- (1, 3);
    \draw[ultra thick,black] (1, 3) -- (2, 4);
    \end{tikzpicture}
    &
    \begin{array}{cc}
    \begin{tikzpicture}[scale=.5,baseline={(0,1)}]
    %
    \draw[dashed,gray] (0,0) grid (5,4);
    \draw[ultra thick,blue] (3, 0) -- (3, 1);
    \draw[ultra thick,blue] (3, 1) -- (4, 2);
    \draw[ultra thick,blue] (4, 2) -- (4, 3);
    \draw[ultra thick,blue] (4, 3) -- (5, 4);
    \draw[ultra thick,red] (2, 0) -- (2, 1);
    \draw[ultra thick,red] (2, 1) -- (3, 2);
    \draw[ultra thick,red] (3, 2) -- (4, 3);
    \draw[ultra thick,red] (4, 3) -- (4, 4);
    \draw[ultra thick,Dandelion] (1, 0) -- (2, 1);
    \draw[ultra thick,Dandelion] (2, 1) -- (2, 2);
    \draw[ultra thick,Dandelion] (2, 2) -- (3, 3);
    \draw[ultra thick,Dandelion] (3, 3) -- (3, 4);
    \draw[ultra thick,YellowGreen] (0, 0) -- (1, 1);
    \draw[ultra thick,YellowGreen] (1, 1) -- (1, 2);
    \draw[ultra thick,YellowGreen] (1, 2) -- (1, 3);
    \draw[ultra thick,YellowGreen] (1, 3) -- (2, 4);
    \node at (2.5,-.5) {$1$} ;
    \end{tikzpicture}
    &
    \begin{tikzpicture}[scale=.5,baseline={(0,1)}]
    %
    \draw[dashed,gray] (0,0) grid (5,4);
    \draw[ultra thick,red] (3, 0) -- (3, 1);
    \draw[ultra thick,red] (3, 1) -- (4, 2);
    \draw[ultra thick,red] (4, 2) -- (4, 3);
    \draw[ultra thick,blue] (4, 3) -- (5, 4);
    \draw[ultra thick,blue] (2, 0) -- (2, 1);
    \draw[ultra thick,blue] (2, 1) -- (3, 2);
    \draw[ultra thick,blue] (3, 2) -- (4, 3);
    \draw[ultra thick,red] (4, 3) -- (4, 4);
    \draw[ultra thick,Dandelion] (1, 0) -- (2, 1);
    \draw[ultra thick,Dandelion] (2, 1) -- (2, 2);
    \draw[ultra thick,Dandelion] (2, 2) -- (3, 3);
    \draw[ultra thick,Dandelion] (3, 3) -- (3, 4);
    \draw[ultra thick,YellowGreen] (0, 0) -- (1, 1);
    \draw[ultra thick,YellowGreen] (1, 1) -- (1, 2);
    \draw[ultra thick,YellowGreen] (1, 2) -- (1, 3);
    \draw[ultra thick,YellowGreen] (1, 3) -- (2, 4);
    \node at (2.5,-.5) {$\s_1$} ;
    \end{tikzpicture}
    \\
    \begin{tikzpicture}[scale=.5,baseline={(0,1)}]
    %
    \draw[dashed,gray] (0,0) grid (5,4);
    \draw[ultra thick,blue] (3, 0) -- (3, 1);
    \draw[ultra thick,blue] (3, 1) -- (4, 2);
    \draw[ultra thick,blue] (4, 2) -- (4, 3);
    \draw[ultra thick,blue] (4, 3) -- (5, 4);
    \draw[ultra thick,Dandelion] (2, 0) -- (2, 1);
    \draw[ultra thick,red] (2, 1) -- (3, 2);
    \draw[ultra thick,red] (3, 2) -- (4, 3);
    \draw[ultra thick,red] (4, 3) -- (4, 4);
    \draw[ultra thick,red] (1, 0) -- (2, 1);
    \draw[ultra thick,Dandelion] (2, 1) -- (2, 2);
    \draw[ultra thick,Dandelion] (2, 2) -- (3, 3);
    \draw[ultra thick,Dandelion] (3, 3) -- (3, 4);
    \draw[ultra thick,YellowGreen] (0, 0) -- (1, 1);
    \draw[ultra thick,YellowGreen] (1, 1) -- (1, 2);
    \draw[ultra thick,YellowGreen] (1, 2) -- (1, 3);
    \draw[ultra thick,YellowGreen] (1, 3) -- (2, 4);
    \node at (2.5,-.5) {$\s_2$} ;
    \end{tikzpicture}
    &
    \begin{tikzpicture}[scale=.5,baseline={(0,1)}]
    %
    \draw[dashed,gray] (0,0) grid (5,4);
    \draw[ultra thick,red] (3, 0) -- (3, 1);
    \draw[ultra thick,red] (3, 1) -- (4, 2);
    \draw[ultra thick,red] (4, 2) -- (4, 3);
    \draw[ultra thick,blue] (4, 3) -- (5, 4);
    \draw[ultra thick,Dandelion] (2, 0) -- (2, 1);
    \draw[ultra thick,blue] (2, 1) -- (3, 2);
    \draw[ultra thick,blue] (3, 2) -- (4, 3);
    \draw[ultra thick,red] (4, 3) -- (4, 4);
    \draw[ultra thick,blue] (1, 0) -- (2, 1);
    \draw[ultra thick,Dandelion] (2, 1) -- (2, 2);
    \draw[ultra thick,Dandelion] (2, 2) -- (3, 3);
    \draw[ultra thick,Dandelion] (3, 3) -- (3, 4);
    \draw[ultra thick,YellowGreen] (0, 0) -- (1, 1);
    \draw[ultra thick,YellowGreen] (1, 1) -- (1, 2);
    \draw[ultra thick,YellowGreen] (1, 2) -- (1, 3);
    \draw[ultra thick,YellowGreen] (1, 3) -- (2, 4);
    \node at (2.5,-.5) {$\s_1\s_2$} ;
    \end{tikzpicture}
    \end{array}
    \end{array}\]
    \caption{On the left, a $((2,2,2,2), \emptyset)$-subnetwork $H$.
      On the right, the four path families that cover $H$, with different colors indicating distinct paths. Below each path family, its permutation type is displayed. Summing over all types, we obtain $\beta(H) = 1 + \s_1 + \s_2 + \s_1\s_2 = (1+\s_1)(1+\s_2)$.
    }
    \label{fig:subnetwork and beta}
\end{figure}
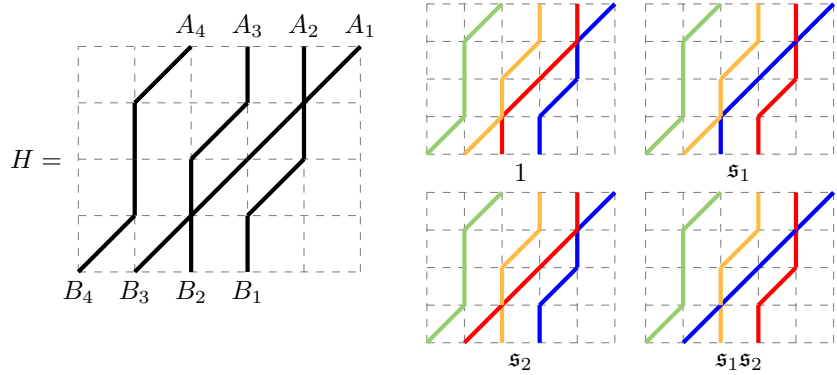

When computing the $f$-immanant of a path matrix, we extend the function $f: \mathfrak{S}_n \to \mathbb{R}$ by linearity so that it is defined on all of $\CC[\mathfrak{S}_n]$.
Then, we may collect the terms in \eqref{eq: first imm expansion} according to the 
subnetwork covered by each path family.
For dual Jacobi--Trudi matrices, we thus obtain the expression
\begin{equation}
\imm_f \left( \tJT_{\lambda/\mu}(\vec x) \right) = \sum_{H} f(\beta(H)) \wt_{\vec x}(H),\label{eq: second det expansion}
\end{equation}
where the sum ranges over all $(\lambda, \mu)$-subnetworks $H$.

Our goal in the remainder of this section is to 
recall Rhoades and Skandera's result, 
which gives a manifestly positive interpretation of \eqref{eq: second det expansion}
in the case of TL immanants.
The primary idea is to show that $f_\tau(\beta(H)) \geq 0$
for any Kauffman diagram $\tau$ and any subnetwork $H$. 

Two south-southwest lattice paths $p$ and $q$ are said to be \emph{noncrossing} if there do not exist $(a_1, b_1), \, (a_2, b_2) \in p$ and $(c_1, d_1), \, (c_2, d_2) \in q$ such that \[a_1 < c_1, \, b_1=d_1, \quad \text{and} \quad c_2 < a_2, \, b_2=d_2.\]
Note that a pair of noncrossing paths may still intersect. 
A path family is said to be \emph{noncrossing} if its paths are pairwise noncrossing.
Observe that, for any subnetwork $H$, there is a unique noncrossing path family $\vec{p}$ that covers $H$. See \Cref{fig:3-2 avoiding network} for an example.

Let \(\vec{p} = (p_1, \ldots, p_n)\) be the unique noncrossing path family  for \(H\). An \textit{essential intersection point} of \(H\) is a vertex in the intersection of two paths \(p_i\) and \(p_{i+1}\) with the property that the previous vertex in each path is not in the intersection of \(p_{i}\) and \(p_{i+1}\).
Note that a pair of paths may have multiple essential intersection points. 
An ordering $(u_1, \ldots, u_\ell)$ of the essential intersection points of $H$ is called \textit{compatible} if there is no directed path in $H$ from $u_b$ to $u_a$ whenever $a < b$.
Given a compatible ordering of the essential intersection points of $H$, an \emph{intersection sequence} for $H$ is $(i_1, \ldots, i_\ell)$, where $i_k=c$ if $u_k$ is in the intersection of $p_c$ and $p_{c+1}$.
For example, the subnetwork in \Cref{fig:subnetwork and beta} has $2$ essential intersection points and a unique intersection sequence $(1,2)$.

A \emph{$(\lambda,\mu)$-generalized wiring diagram} is a $(\lambda,\mu)$-subnetwork $H$ covered by a path family in which no three paths share a vertex.
As with subnetworks, we call $H$ a \emph{generalized wiring diagram}
if it is a $(\lambda,\mu)$-generalized wiring diagram for some $\lambda/\mu$.
Note that, if two path families $\vec{p}$ and $\vec{p}'$ both cover $H$, then $\vec{p}$ has three paths that share a vertex if and only if $\vec{p}'$ does.
This implies that whether a given subnetwork is a generalized wiring diagram is independent of the choice of the covering path family.
Combining the arguments in the proof of \cite[Lemma~2.5]{Rhoades2005} and \cite[p.~461]{Rhoades2005},
we obtain the following lemma.

\begin{lem}\label{lem:beta for wiring diagrams}
Let $H$ be a subnetwork. If $H$ is a generalized wiring diagram, then
    \[\beta(H) = (1+\s_{i_1}) \cdots (1+\s_{i_\ell}),\]
where $(i_1, \ldots, i_\ell)$ is an intersection sequence for $H$.
If $H$ is not a generalized wiring diagram, then $\beta(H)$ is in the kernel of $\theta$.
\end{lem}

Note that, if $H$ is indeed a generalized wiring diagram, then by \Cref{lem:beta for wiring diagrams},
\begin{equation}\label{eq: getting psi algebraically}
    \theta(\beta(H)) = \theta((1+\s_{i_1}) \cdots (1+\s_{i_\ell})) = t_{i_1} \cdots t_{i_\ell}. 
\end{equation}
Using the defining relations of $\TL_n(2)$, the latter expression can
be reduced to \(2^k \tau\) for some integer \( k \) and a Kauffman
diagram \(\tau \in \K_n\). Note that \( k \) and \( \tau \) are
uniquely determined by \( H \), because the element
\(\theta(\beta(H)) \in \TL_n(2)\) has a unique expansion in the basis
\(\K_n\).

\begin{defn}\label{def:1}
  For a generalized wiring diagram \( H \), we define
  \( \epsilon(H) \) and \( \psi(H) \) as the unique integer and the
  Kauffman diagram, respectively, satisfying
\[
  \theta(\beta(H)) = 2^{\epsilon(H)} \psi(H).
\]
The Kauffman diagram \(\psi(H)\) is called the \emph{Temperley--Lieb
  type} of \(H\).
\end{defn}

We can compute $\psi(H)$ and $\epsilon(H)$ visually using the following method \cite[p.~462]{Rhoades2005}. An example is given in \Cref{fig:visually obtaining psi}.
\begin{enumerate}
    \item Contract each doubly covered edge into a single vertex.
    \item Split each vertex shared by two paths into two vertices---one incident to the two incoming edges and the other incident to the two outgoing edges.
    \item If any closed loops are created during steps (1) and (2), remove them.
    \item The number of loops removed in step (3) is $\epsilon(H)$, and the resulting graph is the Kauffman diagram $\psi(H)$.
\end{enumerate}

\begin{figure}
    \centering
    \[\begin{array}{ccccc}
    \begin{tikzpicture}[scale=.75,baseline={(0,1.4)}]
    \node[above] at (5,4) {$A_1$};
    \node[above] at (4,4) {$A_2$};
    \node[above] at (3,4) {$A_3$};
    \node[above] at (2,4) {$A_4$};
    \node[below] at (3,0) {$B_1$};
    \node[below] at (2,0) {$B_2$};
    \node[below] at (1,0) {$B_3$};
    \node[below] at (0,0) {$B_4$};
    \draw[dashed,gray] (0,0) grid (5,4);
    \draw[thick,black] (3, 0) -- (3, 1);
    \draw[thick,black] (3, 1) -- (4, 2);
    \draw[thick,black] (4, 2) -- (4, 3);
    \draw[thick,black] (4, 3) -- (5, 4);
    \draw[thick,black] (2, 0) -- (2, 1);
    \draw[thick,black] (2, 1) -- (3, 2);
    \draw[thick,black] (3, 2) -- (4, 3);
    \draw[thick,black] (4, 3) -- (4, 4);
    \draw[thick,black] (1, 0) -- (2, 1);
    \draw[thick,black] (2, 1) -- (2, 2);
    \draw[thick,black] (2, 2) -- (3, 3);
    \draw[thick,black] (3, 3) -- (3, 4);
    \draw[thick,black] (0, 0) -- (1, 1);
    \draw[thick,black] (1, 1) -- (1, 2);
    \draw[thick,black] (1, 2) -- (1, 3);
    \draw[thick,black] (1, 3) -- (2, 4);
    \end{tikzpicture}
    &
    \leadsto
    &
    \begin{tikzpicture}[scale=.75,baseline={(0,1.4)}]
    \draw[circle, fill=black](5,4) circle[radius = 1mm] node {};
    \draw[circle, fill=black](4,4) circle[radius = 1mm] node {};
    \draw[circle, fill=black](3,4) circle[radius = 1mm] node {};
    \draw[circle, fill=black](2,4) circle[radius = 1mm] node {};
    \draw[circle, fill=black](3,0) circle[radius = 1mm] node {};
    \draw[circle, fill=black](2,0) circle[radius = 1mm] node {};
    \draw[circle, fill=black](1,0) circle[radius = 1mm] node {};
    \draw[circle, fill=black](0,0) circle[radius = 1mm] node {};
    \draw[dashed,gray] (0,0) grid (5,4);
    \draw[thick,black] (3, 0) -- (3, 1);
    \draw[thick,black] (3, 1) -- (4, 2);
    \draw[thick,black] (4, 2) -- (4, 2.75);
    \draw[thick,black] (4, 4) -- (4, 3.25);
    \draw[thick,black] (4, 3.25) -- (5,4);
    \draw[thick,black] (2, 1.25) -- (3, 2);
    \draw[thick,black] (3, 2) -- (4, 2.75);
    \draw[thick,black] (2, 0) -- (2, .75);
    \draw[thick,black] (1, 0) -- (2, .75);
    \draw[thick,black] (2, 1.25) -- (2, 2);
    \draw[thick,black] (2, 2) -- (3, 3);
    \draw[thick,black] (3, 3) -- (3, 4);
    \draw[thick,black] (0, 0) -- (1, 1);
    \draw[thick,black] (1, 1) -- (1, 2);
    \draw[thick,black] (1, 2) -- (1, 3);
    \draw[thick,black] (1, 3) -- (2, 4);
    \end{tikzpicture}
    &
    \leadsto
    &
    \begin{array}{c}
    \psi(H) = 
    \vcenter{\hbox{\begin{tikzpicture}[baseline,x=.75cm, y=.5cm,scale=.75]
    \draw[circle, fill=black](0, 0) circle[radius = .75mm] node {};
    \draw[circle, fill=black](0, -1) circle[radius = .75mm] node {};
    \draw[circle, fill=black](0, -2) circle[radius = .75mm] node {};
    \draw[circle, fill=black](0, 1) circle[radius = .75mm] node {};
    \draw[circle, fill=black](1, -2) circle[radius = .75mm] node {};
    \draw[circle, fill=black](1, -1) circle[radius = .75mm] node {};
    \draw[circle, fill=black](1, 0) circle[radius = .75mm] node {};
    \draw[circle, fill=black](1, 1) circle[radius = .75mm] node {};
    \draw[thick] (0, 1) to (1, 1);
    \draw[thick] (0, 0) to (1, -2);
    \draw[thick] (0, -1) to [bend left=45] (0, -2);
    \draw[thick] (1, 0) to [bend right=45] (1, -1);
    \end{tikzpicture}}}
    = t_1t_2,
    \\ \, \\
    \epsilon(H) = 0
    \end{array}
    \end{array}\]
    \caption{Using the subnetwork $H$ from \Cref{fig:subnetwork and beta}, we implement the visual method for obtaining $\psi(H)$ and $\epsilon(H)$.}
    \label{fig:visually obtaining psi}
\end{figure}
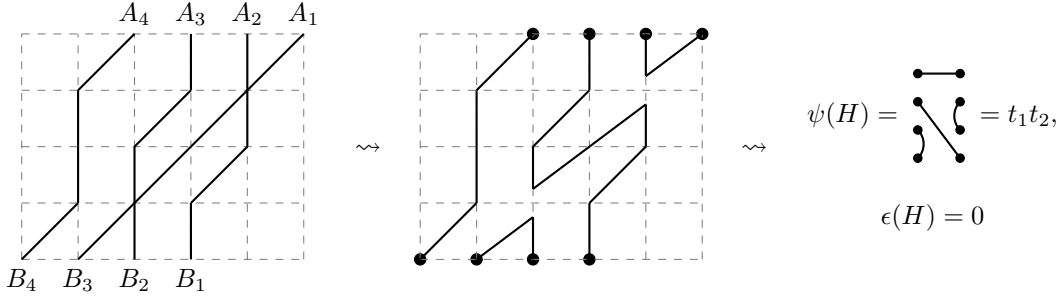

When computing the Temperley--Lieb immanant of Jacobi--Trudi matrices
via \eqref{eq: second det expansion}, one must compute
$f_\tau(\beta(H))$ as $H$ ranges over all subnetworks. Since
\( f_\tau(\beta(H)) \) is the coefficient of \( \tau \) in
\( \theta(\beta(H)) \), by \Cref{lem:beta for wiring diagrams},
Equation~\eqref{eq: getting psi algebraically}, and the ensuing
discussion, we have
\begin{equation}\label{eq:2}
  f_\tau(\beta(H)) 
  = \begin{cases}
    2^{\epsilon(H)} & \text{if } H \text{ is a generalized wiring diagram with } \psi(H)=\tau, \\
    0 & \text{otherwise.}
  \end{cases}
\end{equation}
As a result, we obtain the following corollary.

\begin{cor}\label{cor:TL numerical positivity}\cite[Theorem~3.1]{Rhoades2005}
Given any Kauffman diagram $\tau$ and skew shape $\lambda/\mu$, we have
    \[\imm_\tau(\tJT_{\lambda/\mu}(\vec x)) = \sum_H 2^{\epsilon(H)} \wt_{\vec x}(H),\]
where the sum is over all $(\lambda, \mu)$-generalized wiring diagrams \( H \) with $\psi(H)=\tau$.
\end{cor}

When $\tau$ is the diagram associated to the identity element in $\TL_n(2)$, the requirement that $\psi(H) = \tau$ is equivalent to requiring that $H$ can only be covered by a nonintersecting path family; thus we recover \eqref{eq:LGV}. Otherwise, $\imm_\tau$ is a generating function for subnetworks with a Temperley--Lieb type specified by $\tau$.

\subsection{Representation theory of \texorpdfstring{$\mathfrak{S}_n\times \mathfrak{S}_n$}{Sn x Sn}}

The irreducible representations of the symmetric group
\( \mathfrak{S}_n \) correspond bijectively to the partitions
\( \lambda \) of \( n \). We denote by \( S^\lambda \) the
corresponding irreducible representation, known as the \emph{Specht
  module}. The \emph{Frobenius characteristic map} is an isomorphism
between the space of (virtual) \( \mathfrak{S}_n \)-modules and the space of
symmetric functions of homogeneous degree \( n \), defined as follows:
for a finite-dimensional \( \mathfrak{S}_n \)-module
\( V = \bigoplus_{\lambda \vdash n} (S^\lambda)^{\oplus c_\lambda} \),
\[
  \Frob (V) := \sum_{\lambda \vdash n} c_\lambda s_\lambda.
\]

The irreducible representations of the product
\(\mathfrak{S}_n\times \mathfrak{S}_n\) of two symmetric groups are
indexed by the pairs \( (\lambda,\mu) \) of partitions of \( n \),
with corresponding representations \( S^{\lambda} \otimes S^{\mu} \).
The Frobenius characteristic map extends naturally to the
representations of \(\mathfrak{S}_n\times \mathfrak{S}_n\). More
precisely, for a finite dimensional
\( \mathfrak{S}_n\times \mathfrak{S}_n \)-module
\[
  V = \bigoplus_{\lambda,\mu\vdash n} \left( S^\lambda \otimes S^\mu
  \right)^{\oplus c_{\lambda,\mu}},
\]
we define
\[
  \Frob_{\vec x, \vec y}(V)
  =\sum_{\lambda,\mu\vdash n} c_{\lambda,\mu} s_\lambda(\vec x) s_\mu(\vec y).
\]
In particular, if \( \Frob_{\vec x,\vec y}(V) = f(\vec x, \vec y)\)
for some \( \mathfrak{S}_n \times \mathfrak{S}_n \)-module \( V \),
then \( f(\vec x, \vec y) \) is \( s\)-positive.

\section{Hadamard products of ribbon-like Jacobi--Trudi matrices}\label{Sec: main proofs for 3x2 avoiders}

In \Cref{sec:step1 m expansion}, we prove \Cref{thm: TL pos for 3x2 avoiding}, i.e., that the Temperley--Lieb immanant
    \[\imm_\tau \left(
    \tJT_{\lambda^{(1)} / \mu^{(1)}} (\vec{x}^{(1)}) 
    * \cdots * 
    \tJT_{\lambda^{(k)} / \mu^{(k)}} (\vec{x}^{(k)}) 
    \right)\]
is $s$-positive for any $\tau \in \K_n$, provided that each $\lambda^{(i)} / \mu^{(i)}$ is a skew shape that does not contain a $3 \times 2$ block of cells. We say that such shapes are $(3 \times 2)$-\emph{avoiding}.

If each of the skew shapes is in fact a \emph{ribbon}, i.e., a connected skew shape not containing a $2 \times 2$ block, we can explicitly describe the expansion of this immanant into the Schur basis using standard Young tableaux with prescribed descent sets.
This is done in \Cref{sec:step2 s expansion}.

In \Cref{sec: limitations}, we discuss an extension of the result to a slightly larger class of collections of skew shapes. We then demonstrate that extending our main result to skew shapes containing a $3 \times 2$ block would require a new technique.

\subsection{Hadamard products of \texorpdfstring{$(3 \times 2)$}{3x2}-avoiding skew shapes}\label{sec:step1 m expansion}

When computing an immanant of a single Jacobi--Trudi matrix, one first expands the immanant as a sum over $(\lambda, \mu)$-path families as in \eqref{eq: first det expansion}.
For Hadamard products of several Jacobi--Trudi matrices, we start in an analogous fashion.
Given skew shapes $\lambda^{(1)}/\mu^{(1)}, \ldots, \lambda^{(k)}/\mu^{(k)}$
with \( \ell(\lambda^{(i)})\le n \) for all \( i \), we have
    \begin{align}
    \imm_f \left( \tJT_{\lambda^{(1)}/\mu^{(1)} }(\vec{x}^{(1)})
    *\cdots * 
    \right.\vphantom{.} &
    \left.\vphantom{.} 
    \tJT_{\lambda^{(k)}/\mu^{(k)} }(\vec{x}^{(k)})  \right)
    \nonumber 
    \\ 
    &= \sum_{\sigma \in \mathfrak{S}_n} \,\,
    \sum_{(\vec{p}^{(1)}, \ldots, \vec{p}^{(k)} )} \,\,
    f(\sigma) 
    \wt_{\vec{x}^{(1)}}(\vec{p}^{(1)}) 
    \cdots 
    \wt_{\vec{x}^{(k)}}(\vec{p}^{(k)}).
    \label{eq: hadamard first expansion}
    \end{align}
Here the inner sum is over all tuples $(\vec{p}^{(1)}, \ldots, \vec{p}^{(k)})$
such that each $\vec{p}^{(i)}$ is a $\left(\lambda^{(i)}, \mu^{(i)}\right)$-path family satisfying $\type(\vec{p}^{(i)}) = \sigma$.

In the single matrix case, the next step is to collect path families according to their union, and reorganize the sum to be over subnetworks instead; see \eqref{eq: second det expansion}.
We now introduce the corresponding notions for Hadamard products, which will allow us to simplify the expression above.

Given tuples
\[\underline{\lambda} = (\lambda^{(1)}, \ldots, \lambda^{(k)})
\quad \text{ and } \quad
\underline{\mu} = (\mu^{(1)}, \ldots, \mu^{(k)})\]
of partitions satisfying $\mu^{(i)} \subseteq \lambda^{(i)}$ for all $i$, 
a $(\underline{\lambda},\underline{\mu})$-\emph{multinetwork} is a tuple 
\[\vec{H} = (H^{(1)}, \ldots, H^{(k)})\]
in which each $H^{(i)}$ is a $(\lambda^{(i)}, \mu^{(i)})$-subnetwork.
Given such a multinetwork,
we define the corresponding \emph{path family generating function} to be
    \begin{align}
    \beta(\vec{H}) &= \sum_{\underline{\vec{p}}}
    \type(\underline{\vec{p}}) \in \CC[\mathfrak{S}_n]. \label{eq:def of multibeta}
    \end{align}
Here, the sum is over all tuples $\underline{\vec{p}} = (\vec{p}^{(1)}, \ldots, \vec{p}^{(k)})$ in which each $\vec{p}^{(i)}$ is a $(\lambda^{(i)},\mu^{(i)})$-path family covering $H^{(i)}$, and all of the path families $\vec{p}^{(i)}$ are of the same permutation type. We denote this shared permutation type by $\type(\underline{\vec{p}})$. 
In such a situation, we say that the tuple $\underline{\vec{p}}$ \emph{covers} $\vec{H}$.
An example of a multinetwork is given in \Cref{fig:example of a multinetwork}.

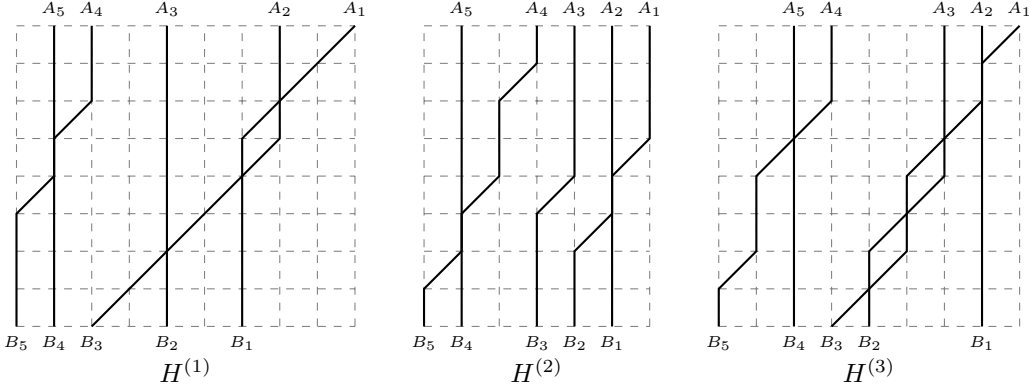
\begin{figure}
    \centering
    \[\begin{array}{ccc}
    \begin{tikzpicture}[scale=.5,baseline={(0,1)}]
    \draw[dashed,gray] (0,0) grid (9, 8);
    \node[above] at (9, 8) {\tiny $A_1$};
    \node[above] at (7, 8) {\tiny $A_2$};
    \node[above] at (4, 8) {\tiny $A_3$};
    \node[above] at (2, 8) {\tiny $A_4$};
    \node[above] at (1, 8) {\tiny $A_5$};
    \node[below] at (6, 0) {\tiny $B_1$};
    \node[below] at (4, 0) {\tiny $B_2$};
    \node[below] at (2, 0) {\tiny $B_3$};
    \node[below] at (1, 0) {\tiny $B_4$};
    \node[below] at (0, 0) {\tiny $B_5$};
    \draw[thick] (9,8) -- (7, 6) -- (7,5) -- (6,4) -- (6,0);
    \draw[thick] (7,8) -- (7, 6) -- (6,5) -- (6,4) -- (4,2) -- (4,0);
    \draw[thick] (4,8) -- (4,2) -- (2,0);
    \draw[thick] (2,8) -- (2,6) -- (1,5) -- (1,0);
    \draw[thick] (1,8) -- (1,4) -- (0,3) -- (0,0);
    \end{tikzpicture}
    &
    \begin{tikzpicture}[scale=.5,baseline={(0,1)}]
    \draw[dashed,gray] (0,0) grid (6, 8);
    \node[above] at (6, 8) {\tiny $A_1$};
    \node[above] at (5, 8) {\tiny $A_2$};
    \node[above] at (4, 8) {\tiny $A_3$};
    \node[above] at (3, 8) {\tiny $A_4$};
    \node[above] at (1, 8) {\tiny $A_5$};
    \node[below] at (5, 0) {\tiny $B_1$};
    \node[below] at (4, 0) {\tiny $B_2$};
    \node[below] at (3, 0) {\tiny $B_3$};
    \node[below] at (1, 0) {\tiny $B_4$};
    \node[below] at (0, 0) {\tiny $B_5$};
    \draw[thick] (6,8) -- (6,5) -- (5,4) -- (5,0);
    \draw[thick] (5,8) -- (5,3) -- (4,2) -- (4,0);
    \draw[thick] (4,8) -- (4,4) -- (3,3) -- (3,0);
    \draw[thick] (3,8) -- (3,7) -- (2,6) -- (2,4) -- (1,3) -- (1,0);
    \draw[thick] (1,8) -- (1,2) -- (0,1) -- (0,0);
    \end{tikzpicture}
    &
    \begin{tikzpicture}[scale=.5,baseline={(0,1)}]
    \draw[dashed,gray] (0,0) grid (8, 8);
    \node[above] at (8, 8) {\tiny $A_1$};
    \node[above] at (7, 8) {\tiny $A_2$};
    \node[above] at (6, 8) {\tiny $A_3$};
    \node[above] at (3, 8) {\tiny $A_4$};
    \node[above] at (2, 8) {\tiny $A_5$};
    \node[below] at (7, 0) {\tiny $B_1$};
    \node[below] at (4, 0) {\tiny $B_2$};
    \node[below] at (3, 0) {\tiny $B_3$};
    \node[below] at (2, 0) {\tiny $B_4$};
    \node[below] at (0, 0) {\tiny $B_5$};
    \draw[thick] (8,8) -- (7,7) -- (7,4) -- (7,0);
    \draw[thick] (7,8) -- (7,6) -- (6,5) -- (6,4) -- (5,3) -- (5,2) -- (4,1) -- (4,0);
    \draw[thick] (6,8) -- (6,5) -- (5,4) -- (5,3) -- (4,2) -- (4,1) -- (3,0);
    \draw[thick] (3,8) -- (3,6) -- (2,5) -- (2,0);
    \draw[thick] (2,8) -- (2,5) -- (1,4) -- (1,2) -- (0,1) -- (0,0);
    \end{tikzpicture}
    \\
    H^{(1)}
    & 
    H^{(2)}
    & 
    H^{(3)}
    \end{array} \]
    \caption{A $(\underline{\lambda}, \underline{\mu})$-multinetwork $\vec{H}$ for
    \begin{align*}
    \underline{\lambda} &= ((5, 4, 2, 1, 1), (2, 2, 2, 2, 1), (4, 4, 4, 2, 2)), \\
    \underline{\mu} &= ((2,1), (1, 1, 1), (3, 1, 1, 1)).
    \end{align*}
    Note that $\beta(\vec{H}) = 2^3 (1+\s_1)(1+\s_4) = 8 \cdot \B{\{1,4\}}.$ 
    For each collection $\vec p$ of path families covering $\vec H$,
    swapping paths around each loop does not affect the permutation type. In particular, for each permutation type, there are $2^{3}$ path families of that type covering $\vec{H}$.  
    }
    \label{fig:example of a multinetwork}
\end{figure}

Let \( \underline{\vec x} = (\vec x^{(1)},\dots,\vec x^{(k)}) \).
For any tuple $\underline{\vec{p}} = (\vec{p}^{(1)}, \ldots, \vec{p}^{(k)})$ covering a multinetwork $\vec{H}$,
we define the $\underline{\vec x}$-weight of $\vec{H}$ to be
\[\wt_{\underline{\vec x}}(\vec H) = \wt_{\vec{x}^{(1)}}(\vec{p}^{(1)}) \cdots \wt_{\vec{x}^{(k)}}(\vec{p}^{(k)}).\]
We can use this notation to group tuples of path families
according to the multinetwork that they cover and rewrite the
expansion in \eqref{eq: hadamard first expansion} as
    \begin{equation}
    \imm_f \left( \tJT_{\lambda^{(1)}/\mu^{(1)}}(\vec{x}^{(1)}) *\cdots * \tJT_{\lambda^{(k)}/\mu^{(k)}}(\vec{x}^{(k)}) \right) = \sum_{\vec{H}} f(\beta(\vec{H})) \wt_{\underline{\vec x}}(\vec{H}), \label{eq: hadamard second expansion}
    \end{equation}
where the sum is over all $(\underline{\lambda}, \underline{\mu})$-multinetworks $\vec{H}$.

When showing that a Temperley--Lieb immanant of a single Jacobi--Trudi matrix is nonnegative,
the final step is to argue that for each subnetwork $H$, the path family generating functions $\beta(H)$ have a simple form, for which the evaluations $f_\tau(\beta(H))$ are clearly nonnegative. See \Cref{lem:beta for wiring diagrams} and \eqref{eq:2}.
In what follows, we argue that $\beta(\vec{H})$ still has a simple form
if $\vec{H}$ is a multinetwork for a collection of $(3 \times 2)$-avoiding skew shapes.

To this end, we introduce some notation.
Given $I \subseteq [n-1]$, define
    \[\B{I} = \prod_{i \in I} (1+\s_i) \in \CC[\sn].\]
Throughout, products over subsets are taken in increasing order from left to right.
We also define \(\T{I} = \theta(\B{I})\). Note that
\begin{align*}
  \T{I} = \theta \left( \prod_{i \in I} (1+\s_i) \right)
  = \prod_{i \in I} (1+t_i - 1)
  = \prod_{i \in I} t_i.
\end{align*}
From this, one can see that for \( I,I' \subseteq[n-1] \), we have
\( \tau(I) = \tau(I') \) if and only if \( I=I' \).
Hence, we also have \( \B{I} = \B{I'} \) if and only if \( I=I' \).

Recall \( \epsilon(H) \) and \( \psi(H) \) defined in
\Cref{def:1}.

\begin{lem}\label{lem:beta for ribbons}
  Suppose \(\lambda / \mu\) is a \((3 \times 2)\)-avoiding skew shape
  with \( \ell(\lambda)\le n \). Let \( H \) be a
  \((\lambda, \mu)\)-subnetwork, and let
  \( \vec{p}=(p_1, \ldots, p_n) \) be the unique noncrossing path
  family covering \( H \). Then the following hold:
  \begin{enumerate}
  \item \(H\) is a generalized wiring diagram.
\item There exists a unique set \(I \subseteq [n-1]\) such that
  \(\beta(H) = 2^{\epsilon(H)}\B{I}\). In particular,
  \( \psi(H) = \tau(I) \).
\item The path \(p_i\) intersects \(p_{i+1}\) if and only if
  \(i \in I\).
\item The essential intersection points of \(p_i\) with \(p_{i+1}\)
  always occur before the essential intersection points of \(p_{i+1}\)
  with \(p_{i+2}\).
  \end{enumerate}
\end{lem}

\begin{proof}
  For the first statement, we need to show that there is no vertex \(v\)
  contained in three of the paths in \(\vec{p}\). Since \( \vec{p} \) is
  a noncrossing path family, it suffices to show that no vertex \(v\) is
  contained in three consecutive paths \(p_i, p_{i+1},\) and \(p_{i+2}\).
  Suppose, for the sake of contradiction, that \(v\) is such a vertex.

Since the vertex \(v=(x,y)\) is contained in the south-southwest lattice path \(p_i: A(\lambda)_i \to B(\mu)_i\), it must lie weakly to the right of \(B(\mu)_i = (\mu_i + n - i, 0)\) in the plane, which implies \(\mu_i + n-i \leq x.\)
Similarly, since \(v\) is contained in \(p_{i+2}\), it must lie weakly to the left of \(A(\lambda)_{i+2} = (\lambda_{i+2} + n-i-2, \infty)\), which implies \(x \leq \lambda_{i+2} + n-i-2.\)
Together, these conditions imply \(\mu_i + n-i \leq \lambda_{i+2} + n-i-2\), equivalently \(\mu_i \leq \lambda_{i+2}-2\).
However, 
since \(\lambda/\mu\) is \((3 \times 2)\)-avoiding, we have
\(\lambda_{i+2}-2 < \mu_{i}\).
Thus we have obtained a contradiction. It follows that the paths \(p_i\) and \(p_{i+2}\) can never intersect, which implies that no three paths in \(\vec{p}\) can share a vertex.

For the remaining statements, we construct a compatible ordering
\((u_1, \ldots, u_\ell)\) of the essential intersection points of
\(H\) for which the corresponding intersection sequence
\((i_1, \ldots, i_\ell)\) is weakly increasing. First, place all
essential intersection points of \(p_1\) and \(p_2\), then all
essential intersection points of \(p_2\) and \(p_3\), and so on. Each
time, place the vertices from the path \(p_i\) in the order in which
they appear in \(p_i\). By construction, the corresponding sequence
\((i_1, \ldots, i_\ell)\) is weakly increasing, so we only need to
check that this sequence is compatible with \(H\), i.e., that there
can never be a directed path in \(H\) from \(u_b\) to \(u_a\) for
\(a < b\). There are two cases to consider. If \(u_b\) and \(u_a\)
both belong to the same path \(p_i\), they have already been ordered
compatibly. Otherwise \(u_{a}\) lies in the intersection of paths
\(p_i\) and \(p_{i+1}\), while \(u_{b}\) belongs to the intersection
of \(p_{j}\) and \(p_{j+1}\) for \(j \geq i+1\). In particular,
\(u_b\) lies on a path \(p_{j+1}\) with \(j+1 \geq i+2\). Following a
similar argument as we did above, this implies \(u_b\) lies strictly
to the left of \(u_a\) in the plane, so there is no directed path from
\(u_b\) to \(u_a\).

Now we are ready to prove the remaining statements. Let
\(I \subseteq [n-1]\) be the set of integers appearing in
\((i_1, \ldots, i_\ell)\). Since \((i_1, \ldots, i_\ell)\) is weakly
increasing, by \Cref{lem:beta for wiring diagrams} and the fact that
\((1+\s_i)^k = 2^{k-1}(1+\s_i)\), we have
\[
  \beta(H) = (1+\s_{i_1})(1+\s_{i_2}) \cdots (1+\s_{i_\ell}) =
  2^\epsilon \prod_{i \in I} (1+\s_i) = 2^\epsilon \B{I},
\]
where \(\epsilon = \ell-|I|\). Since \(i \in I\) if and only if
\(p_i\) and \(p_{i+1}\) intersect, \( I \) and \( \epsilon \) are
uniquely determined by \(H\). Since
\( \theta(\beta(H)) = \theta(2^\epsilon \B{I}) = 2^\epsilon \T{I}\),
we have \( \epsilon=\epsilon(H) \) and \( \T{I} = \psi(H) \). This
shows the second statement. The third and the fourth statements follow
immediately from the construction of the ordering
\( (u_1,\dots,u_\ell) \) above.
\end{proof}

As mentioned in \Cref{sec: imms and tl setup}, \(\epsilon(H)\) is the
number of closed loops in \(H\). 

\begin{exam}\label{ex:computing beta for 3x2 avoider}
  Let $\lambda = (7,6,5,2,2)$ and $\mu=(5,1,1)$. Then the skew
  shape $\lambda / \mu$ is $(3 \times 2)$-avoiding with \( n=\ell(\lambda)=5 \).
  Recalling that $A(\lambda)_i = (\lambda_i + n-i, \infty)$ and
  $B(\mu)_i = (\mu_i + n-i, 0)$, we have
  $$A(\lambda)_1 = (11, \infty), \quad A(\lambda)_2 = (9, \infty), \quad
  A(\lambda)_3 = (7, \infty), \quad A(\lambda)_4 = (3, \infty), \quad
  A(\lambda)_5 = (2, \infty) \quad \text{and}$$
$$B(\mu)_1 = (9, 0), \quad B(\mu)_2 = (4, 0), \quad B(\mu)_3 = (3, 0), \quad B(\mu)_4 = (1, 0), \quad B(\mu)_5 = (0, 0).$$
Let $H$ be the $(\lambda, \mu)$-subnetwork from \Cref{fig:3-2 avoiding network}. Observe that $\epsilon(H)=2$, since there are two closed loops in $H$.
Also observe that
\[\beta(H) = (1+\s_{2})^3(1+\s_{3})(1+\s_{4}) = 2^2 \cdot \B{\{2,3,4\}}. \]
After the initial essential intersection point $(7,9)$ of $p_2$ and $p_3$, the path $p_2$ cannot cross $p_1$. The remaining vertices in $p_2$ lie strictly to the left of $B_1$, which is the destination of the path $p_1$;
see \Cref{fig:3-2 avoiding network}. 
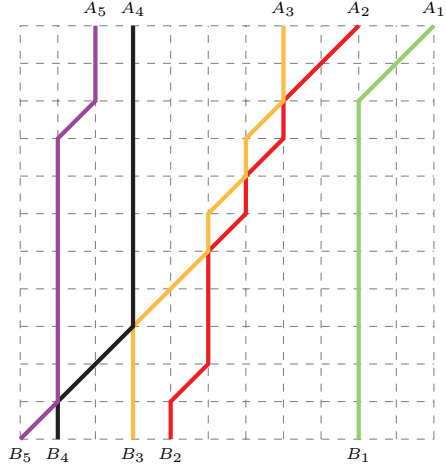
\begin{figure}
    \centering
    \[\begin{array}{ccc}
    \begin{tikzpicture}[scale=.5,baseline={(0,1)}]
    \draw[dashed,gray] (0,0) grid (11, 11);
    \node[above] at (11, 11) {\tiny $A_1$};
    \node[above] at (9, 11) {\tiny $A_2$};
    \node[above] at (7, 11) {\tiny $A_3$};
    \node[above] at (3, 11) {\tiny $A_4$};
    \node[above] at (2, 11) {\tiny $A_5$};
    \node[below] at (9, 0) {\tiny $B_1$};
    \node[below] at (4, 0) {\tiny $B_2$};
    \node[below] at (3, 0) {\tiny $B_3$};
    \node[below] at (1, 0) {\tiny $B_4$};
    \node[below] at (0, 0) {\tiny $B_5$};
    \draw[ultra thick,YellowGreen] (11,11) -- (9, 9) -- (9,0);
    \draw[ultra thick,Red] (9,11) -- (7, 9) -- (7,8) -- (6,7) -- (6,6) -- (5,5) -- (5,2) -- (4,1) -- (4,0);
    \draw[ultra thick,Dandelion] (7,11) -- (7,9) -- (6,8) -- (6,7) -- (5,6) -- (5,5) -- (3,3) -- (3,0);
    \draw[ultra thick,Black] (3,11) -- (3,3) -- (1,1) -- (1,0);
    \draw[ultra thick,Purple] (2,11) -- (2,9) -- (1,8) -- (1,1) -- (0,0);
    \end{tikzpicture}
    \end{array} \]
    \caption{A $((7,6,5,2,2),(5,1,1))$-subnetwork $H$. Paths for the unique noncrossing path family covering $H$ have been highlighted in distinct colors.}
    \label{fig:3-2 avoiding network}
\end{figure}
\end{exam}

\begin{lem}\label{lem: beta for multinetworks}
  Suppose
  $\lambda^{(1)} / \mu^{(1)}, \ldots, \lambda^{(k)} / \mu^{(k)}$ are
  $(3 \times 2)$-avoiding skew shapes, with
  \( \ell(\lambda^{(i)})\le n \) for all \( i \). Let
  $\vec{H} = (H^{(1)}, \ldots, H^{(k)})$ be a tuple such that each
  \( H^{(i)} \) is a \( (\lambda^{(i)},\mu^{(i)}) \)-subnetwork.
  Then 
  \[
    \beta(\vec{H}) = 2^{\epsilon(\vec{H})} \B{I_1 \cap \cdots \cap I_k},
  \]
  where
  $\epsilon(\vec{H}) = \epsilon(H^{(1)}) + \cdots + \epsilon(H^{(k)})$,
  and for all $1 \leq i \leq k$, the set \( I_i \subseteq [n-1] \) is the unique set such that
  \(\beta(H^{(i)}) = 2^{\epsilon(H^{(i)})}\B{I_i}\).
\end{lem}

\begin{proof}
  Let $\star$ be the product in $\CC[\sn]$ defined by
  $u \star w = u\delta_{u,w}$ for $u,w \in \sn$ and extended by
  linearity. By the definition of \( \beta(\vec{H}) \) in
  \eqref{eq:def of multibeta}, we have
\[
  \beta(\vec{H}) = \beta (H^{(1)}) \star \cdots \star
  \beta(H^{(k)}).
\]
By \Cref{lem:beta for ribbons}, for each \( i\in [k] \), we have
$\beta(H^{(i)}) = 2^{\epsilon(H^{(i)})}\B{I_i}$ for a unique set \( I_i \subseteq [n-1] \).
Thus, the above equation can be written as
    \begin{equation}
    \beta(\vec{H}) 
    = 2^{\epsilon(H^{(1)})} \B{I_1} \star  \cdots \star  2^{\epsilon(H^{(k)})} \B{I_k}  
    = 2^{\epsilon(\vec{H})} \left( \B{I_1} \star \cdots \star \B{I_k}\right). \label{eq: union expression}
    \end{equation}

Now observe that for any \( I\subseteq [n-1] \),
\begin{align*}
    \B{I} &= \prod_{i \in I} (1+\s_i)
    = \sum_{A \subseteq I} \s_A,
\end{align*}
where \( \s_A = \prod_{i\in A}\s_i \). As before, products over
subsets are taken in increasing order from left to right.
For any sets
\( A,B\subseteq [n-1] \), we have \( \s_A = \s_B \) if and only if
\( A=B \). This can easily be seen from the observation that
\( \min(A) \) is the smallest integer \( i \) such that
\( \s_A(i)\ne i \). Thus, for any sets $I,I'\subseteq [n-1]$, we have
\begin{align}
\B{I} \star  \B{I'} 
= \left(\sum_{A \subseteq  I} \s_A \right) 
\star  
\left(\sum_{B \subseteq  I'} \s_B \right)
= \sum_{C \subseteq I \cap I'} \s_C 
= \B{I \cap I'}. 
\label{eq:hadamard-is-int}
\end{align}
By \eqref{eq: union expression} and
\eqref{eq:hadamard-is-int}, we obtain the lemma.
\end{proof}

\begin{exam}
Consider the multinetwork presented in \Cref{fig:example of a multinetwork}. In this case, we have
\[\begin{array}{llll}
    \beta(H^{(1)}) &= (1+\s_1)^2(1+\s_2)(1+\s_4) &= 2 \cdot \B{\{1,2,4\}}, & \epsilon(H^{(1)}) = 1; \\
    \beta(H^{(2)}) &= (1+\s_1)(1+\s_4) &= \B{\{1,4\}}, & \epsilon(H^{(2)}) = 0; \\
    \beta(H^{(3)}) &= (1+\s_1)(1+\s_2)^3(1+\s_4) &= 4 \cdot \B{\{1,2,4\}}, & \epsilon(H^{(3)}) = 2.
\end{array}\]
In addition,
\begin{align*}
    \beta(\vec{H}) &= 2^3(1+\s_1)(1+\s_4) \\
    &= 8 \cdot \B{\{1,4\}} \\
    &= \left( 2 \cdot \B{\{1,2,4\}} \right)
    \star
    \B{\{1,4\}}
    \star 
    \left( 4 \cdot \B{\{1,2,4\}} \right).
    \end{align*}
\end{exam}

\begin{thm}\label{thm: TL multinetwork expansion}
  Let \( \tau\in \K_n \) and suppose
  $\lambda^{(1)} / \mu^{(1)}, \ldots, \lambda^{(k)} / \mu^{(k)}$ are
  $(3 \times 2)$-avoiding skew shapes, with
  \( \ell(\lambda^{(i)})\le n \) for all \( i \). If $\tau=\T{I}$ for
  some $I \subseteq [n-1]$, we have
  \[
    \imm_{\tau} \left( \tJT_{\lambda^{(1)} /
        \mu^{(1)}}(\vec{x}^{(1)}) * \cdots * \tJT_{\lambda^{(k)} /
        \mu^{(k)}}(\vec{x}^{(k)}) \right) = \sum_{\vec{H}}
    2^{\epsilon(\vec{H})} \wt_{\underline{\vec x}}(\vec{H}),
  \]
  where the sum is over all $(\underline{\lambda}, \underline{\mu})$-multinetworks
  $\vec{H}$ with $\beta(\vec{H}) = 2^{\epsilon(\vec{H})} \B{I}$.
  If $\tau$ is not of this form, the immanant is $0$.
  
  In particular, this immanant is always $m$-positive.
\end{thm}

\begin{proof}
Using \eqref{eq: hadamard second expansion}, we first write this immanant as 
    \[\imm_{\tau} \left( \tJT_{\lambda^{(1)} / \mu^{(1)}}(\vec{x}^{(1)}) *\cdots *
    \tJT_{\lambda^{(k)} / \mu^{(k)}}(\vec{x}^{(k)}) \right)
    = \sum_{\vec{H}} f_\tau(\beta(\vec{H})) \wt_{\underline{\vec x}}(\vec{H}),\]
where the sum is over all $(\underline{\lambda}, \underline{\mu})$-multinetworks $\vec{H}$.
Recall that $f_\tau(\beta(\vec{H}))$ is the coefficient of $\tau$ in $\theta(\beta(\vec{H}))$. 

Suppose that $\tau=\T{I}$ for some $I \subseteq [n-1]$.
By \Cref{lem: beta for multinetworks}, each $\beta(\vec{H})$ is of the form $2^{\epsilon(\vec{H})} \B{I'}$ for some $I' \subseteq [n-1]$. 
By definition,
\[\theta(2^{\epsilon(\vec{H})} \B{I'}) = 2^{\epsilon(\vec{H})} \T{I'}.\]
Since \( \tau(I) = \tau(I') \) if and only if  \( I=I' \),
the above implies that $f_{\T{I}}(\B{I'}) = 2^{\epsilon(\vec{H})}$ if $I=I'$ and $f_{\T{I}}(\B{I'}) = 0$ otherwise, so the claim follows. 

If \( \tau \) is not of the form \( \T{I} \) with \( I\subseteq [n-1] \), then the above argument also shows that the immanant is \( 0 \).
\end{proof}

We can now obtain \Cref{thm: TL pos for 3x2 avoiding} from the introduction as a corollary.

\begin{cor}\label{cor:all TL imms s positive}
  For any $\tau \in \K_n$ and any $(3 \times 2)$-avoiding skew shapes
  $\lambda^{(1)} / \mu^{(1)}, \ldots, \lambda^{(k)} / \mu^{(k)}$ with
  \( \ell(\lambda^{(i)})\le n \) for all \( i \), the immanant
  \[ \imm_\tau \left( \tJT_{\lambda^{(1)} /
        \mu^{(1)}}(\vec{x}^{(1)}) *\cdots *
      \tJT_{\lambda^{(k)} / \mu^{(k)}}(\vec{x}^{(k)})
    \right)
\] is Schur positive.
\end{cor}

\begin{proof}
  By \Cref{thm: TL multinetwork expansion}, we may assume that \( \tau=\tau(I) \)
  for some \( I\subseteq [n-1] \).
  Recall that a
  \((\underline{\lambda}, \underline{\mu})\)-multinetwork \(\vec{H}\)
  is a tuple \((H^{(1)}, \ldots, H^{(k)})\) in which each \(H^{(i)}\)
  is a \((\lambda^{(i)} ,\mu^{(i)})\)-subnetwork. By \Cref{lem: beta
    for multinetworks}, the condition 
  \(\beta(\vec{H}) = 2^{\epsilon(\vec{H})}\B{I}\) is equivalent to
  \(I_1 \cap \cdots \cap I_k = I\), where
  \(\beta(H^{(i)})=2^{\epsilon(H^{(i)})}\B{I_i}\) for each \( i \).
  Therefore, the equation in \Cref{thm: TL multinetwork expansion} can be
  rewritten as
  \begin{multline}
        \imm_{\tau} \left( \tJT_{\lambda^{(1)} /
        \mu^{(1)}}(\vec{x}^{(1)}) * \cdots * \tJT_{\lambda^{(k)} /
        \mu^{(k)}}(\vec{x}^{(k)}) \right) \\
    = \sum_{\substack{I_1,\dots,I_k \subseteq [n-1] \\ I_1 \cap \cdots \cap I_k = I}} \,
     \prod_{i=1}^k
    \left( \sum_{\substack{H^{(i)} \\ \beta(H^{(i)}) = 2^{\epsilon(H^{(i)})}\B{I_i}}} 
    2^{\epsilon(H^{(i)})} \wt_{\vec{x}^{(i)}}(H^{(i)}) \right) .
    \label{eq:network expansion verbose}
    \end{multline}
    By \Cref{thm: TL multinetwork expansion}, for each $i \in [n]$, we
    have
    \[
       \sum_{\substack{H^{(i)} \\ \beta(H^{(i)}) = 2^{\epsilon(H^{(i)})}\B{I_i}}} 
    2^{\epsilon(H^{(i)})} \wt_{\vec{x}^{(i)}}(H^{(i)}) 
    = \imm_{\T{I_i}} \left( \tJT_{\lambda^{(i)} / \mu^{(i)}}(\vec{x}^{(i)}) \right).\]
By \cite[Proposition 3]{RHOADES2006793} and \cite[Proposition 5 and the paragraph above it]{RHOADES2006793}, each Temperley--Lieb immanant of a single Jacobi--Trudi matrix is Schur positive. Thus the expression in \eqref{eq:network expansion verbose} is Schur positive.
\end{proof}

\subsection{Schur expansion of ribbon Hadamard products}\label{sec:step2 s expansion}

In this subsection, we provide a more explicit description of the Schur expansion from \Cref{cor:all TL imms s positive} in the case where all of the skew shapes $\lambda^{(i)} / \mu^{(i)}$ are ribbons. The description is given in terms of standard Young tableaux with restricted descent sets.

We begin with the following lemma, which guarantees that the powers of
$2$ from \Cref{thm: TL multinetwork expansion} do not appear in the
ribbon case.

\begin{lem}\label{lem: epsilon for ribbons}
    Suppose $\lambda / \mu$ is a ribbon.
    Then any $(\lambda, \mu)$-subnetwork $H$ satisfies $\epsilon(H) = 0$.
\end{lem}

\begin{proof}
  As in the proof of \Cref{lem:beta for ribbons}, let
  \(\vec{p}=(p_1,\dots,p_n)\) be the unique noncrossing path family
  covering \(H\), where \( n \) is an integer with
  \( \ell(\lambda)\le n \). It suffices to show that each pair
  \((p_i, p_{i+1})\) of paths in \(\vec{p}\) has at most one essential
  intersection point, because this implies that there are no loops in \( H \).
  
  To show this, first observe that because \(\lambda / \mu\) is a
  ribbon, it satisfies \(\mu_i = \lambda_{i+1}-1\) for all
  \(a\le i\le b-1\), where \( [a,b] = \{i:\lambda_i>\mu_i\} \).
  Then, let
  \(v = (x,y)\) be an essential intersection point of \(p_i\) and
  \(p_{i+1}\) 
  for some \( i \in [a,b-1]\).
  Note that $i$ lies in this range because the paths $p_k$ for \( k\le a-1 \)
  or \( k\ge b+1 \) are vertical paths consisting of south steps.
  Since \(v\) lies on the path
  \(p_i: A(\lambda)_i \to B(\mu)_i\), it must lie weakly to the right
  of \(B(\mu)_i = (\mu_i + n-i, 0)\) in the plane. Since \(v\)
  lies on the path \(p_{i+1}: A(\lambda)_{i+1} \to B(\mu)_{i+1}\), it
  must lie weakly to left of
  \(A(\lambda)_{i+1} = (\lambda_{i+1} + n-i-1, \infty)\). The latter two
  conditions together imply that
  \[
    \lambda_{i+1} + n-i-1 \leq x \leq \mu_i + n-i.
  \]
  Since \(\mu_i = \lambda_{i+1}-1\), this implies
  \(x = \lambda_{i+1} + n-i-1 = \mu_i + n-i\). Therefore all essential
  intersection points of \(p_i\) and \(p_{i+1}\) must lie on a fixed
  vertical line. Since both paths move only south and southwest in the
  plane, there is at most one essential intersection point.
\end{proof}

We now establish some notation that will be useful in translating between Temperley--Lieb types and standard Young tableaux.
For a
ribbon \( R \) of size \( m \), label the cells sequentially
from \( 1 \) to \( m \), starting from the top-right cell. A
\emph{descent} of \( R \) is an integer \( i\in [m-1] \) such that the
cell labeled \( i \) is positioned immediately above the cell labeled
\( i + 1 \). We denote by \(\Des(R)\) the set of descents of \( R \).
For a word \( u=u_1 \cdots u_n = (u_1,\dots,u_m)\in \ZZ_{\ge1}^m \),
we also define \( \Des(u) \) to be the set of integers
\( i\in [m-1] \) such that \( u_i>u_{i+1} \), and we denote
\( \vec x_u = x_{u_1} \cdots x_{u_m} \).

Suppose that \(R = \lambda / \mu\) is a ribbon
with \( \{i:\lambda_i>\mu_i\} = [a,b] \).
Then
\begin{equation}\label{eq:3}
  \Des(R) = \left\{\sum_{j=1}^{i} (\lambda_j-\mu_j) : i\in [a,b-1]
  \right\}.
\end{equation}
For any set \(I \subseteq \NN\), we define
\begin{equation}\label{eq:d(I,R)}
  d(I,R) = \left\{\sum_{j=1}^{i} (\lambda_j-\mu_j) : i \in I
  \right\}.
\end{equation}
Note that \( d([a,b-1], R) = \Des(R) \).
Lastly, we denote $\NDes(R) := [m-1] \setminus \Des(R).$

\begin{lem}\label{thm:subnetworks-to-words}
  Suppose \(R = \lambda/\mu\) is a ribbon of size \(m\) with
  \( \{i:\lambda_i>\mu_i\} = [a,b] \). Then, for any set
  \(I \subseteq [a,b-1]\), there is a bijection between the set of
  \(R\)-subnetworks \(H\) with \(\beta(H)=\B{I}\) and the set of words
  \(u \in \mathbb{Z}_{\geq 1}^m \) satisfying
  \(\Des(u) = \NDes(R) \cup d(I,R)\).
\end{lem}

\begin{proof}
  Suppose \( \ell(\lambda)\le n \). Let \( X \) be the set of
  \( R \)-subnetworks \( H \) with \(\beta(H)=\B{I}\) and let \( Y \)
  be the set of words \(u \in \mathbb{Z}_{\geq 1}^m \) satisfying
  \(\Des(u) = \NDes(R) \cup d(I,R)\). We need to find a bijection from
  \( X \) to \( Y \).

  Consider \( H\in X \). Let \(\vec{p}=(p_1, \ldots, p_n)\) be the
  unique noncrossing path family that covers \(H\). Next, associate to
  each path \(p_i\) a decreasing word \(u(p_i)\), consisting of
  positive integers, recorded by tracking the \(y\)-coordinates of the
  starting points of southwest steps in \(p_i\), starting from the
  top. We then define \(u(\vec{p}) \in \ZZ^m_{\ge1}\) as the
  concatenation of the words \(u(p_1), \dots, u(p_n)\). We now show
  that \(\Des(u(\vec{p})) = \NDes(R) \cup d(I,R)\).

There are two types of descents in \(\Des(u(\vec{p})) \subseteq [m-1]\) to consider: 
\begin{itemize}
\item[(i)] descents that come from one of the words \(u(p_i)\), or 
\item[(ii)] descents that occur when the last entry of \(u(p_i)\) is greater than the first entry of \(u(p_{i+1})\). 
\end{itemize}
We claim that the descents of type (i) form the set $[m-1] \setminus \Des(R) = \NDes(R)$. 
We also claim that the descents of type (ii) form the set $d(I,R)$.
These two claims together will demonstrate that
\(\Des(u(\vec{p})) = \NDes(R) \cup d(I,R)\), 
as desired.

For the first claim, observe that since each
\(p_i:(\lambda_i+n-i, \infty) \to (\mu_i+n-i, 0)\) is a south-southwest
lattice path, each \(u(p_i)\) is a decreasing word. Since the \(x\)
coordinates of these points differ by \(\lambda_i - \mu_i\), each
\(p_i\) must take exactly \(\lambda_i-\mu_i\) southwest steps, so
\(u(p_i)\) is a word of length \(\lambda_i-\mu_i\). Thus \(u(p_i)\)
contributes the descents to \(u(\vec{p})\) in positions
\( N_{i-1}+1,N_{i-1}+2, \dots,N_{i}-1 \),
where \( N_i = \sum_{j=1}^{i} \left(\lambda_j-\mu_j\right) \).
Thus, by \eqref{eq:3},
the set of descents of type (i) is
\[
  [m-1] \setminus \{N_1,\dots,N_n\} = [m-1] \setminus \{N_a,\dots,N_{b-1}\} = [m-1] \setminus \Des(R),
\]
which shows the first claim.

For the second claim, suppose the last entry of \(u(p_i)\) is greater
than the first entry of \(u(p_{i+1})\). Since \(p_j\) is of length
\(\lambda_j-\mu_j\) for all \(j\), the corresponding descent in
\(u(\vec{p})\) is at position
\( N_i = \sum_{j=1}^{i} (\lambda_j - \mu_j)\). To prove the claim, we
must argue that this occurs if and only if \(i \in I\). Since \(p_i\)
has destination \((\mu_i+n-i, 0)\), the last southwest step of \(p_i\)
terminates with \(x\)-coordinate \(\mu_i+n-i\). Similarly, the first
southwest step of \(p_{i+1}\) originates with \(x\)-coordinate
\(\lambda_{i+1}+n-i-1\). Since \(\mu_i = \lambda_{i+1}-1\), the
vertical lines \(x = \lambda_{i+1} + n-i-1 \) and \(x = \mu_i + n-i\)
are identical. The fact that the last entry of \(u(p_i)\) is greater
than the first entry of \(u(p_{i+1})\) means that the last southwest
step of \(p_i\) originates strictly north of the first southwest step
of \(p_{i+1}\). It follows that \(p_i\) and \(p_{i+1}\) intersect. By
\Cref{lem:beta for ribbons}, since \(\beta(H)=\B{I}\), such an
intersection occurs if and only if \(i \in I\). This implies the
second claim.

By the two claims, \( H\mapsto u \) is a map from \( X \) to
\( Y \). To show that this is a bijection, we construct its inverse.
Observe that a south-southwest path \( p \) is determined by the
decreasing word \( u(p) \) defined earlier. Suppose
\( u = (u_1,\dots,u_n)\in Y \). First, let
\( p_1:A(\lambda)_1\to B(\mu)_1 \) be the unique path with
\( u(p_1) = (u_1,\dots,u_{N_1}) \). Then define
\( p_2:A(\lambda)_2\to B(\mu)_2 \) to be the unique path with
\( u(p_2) = (u_{N_1+1},\dots,u_{N_2}) \), and so on. In this way, we
obtain a sequence \( \vec p = (p_1,\dots,p_n) \) of paths. Let \( H \)
be the subnetwork covered by \( \vec p \). It is easy to check that
\( u\mapsto H \) is the desired inverse map.
  \end{proof}

\begin{exam}
Consider the subnetwork $H$ for the ribbon $R= (2,2,2,2,1)/(1,1,1)$ of size $m=6$ presented in \Cref{fig:subnetwork for a ribbon}.
Observe that $\Des(R) = \{1,2,3,5\}$, so $\NDes(R) = \{4\}$.
In addition, we have $\beta(H) = \B{\{1,4\}}$, so we set $I = \{1,4\}$.
In this case, $d(I,R) = d(\{1,4\}, R) = \{1,5\}$.
Applying the bijection from \Cref{thm:subnetworks-to-words}, we obtain 
    \begin{align*}
        u(p_1) = 4, \, u(p_2) = 2, \, u(p_3) = 3, \, u(p_4) = 63, \text{ and } u(p_5) = 1.
    \end{align*}
Thus $u(\vec{p}) = 4 2 3 63 1.$
We have $\Des(u(\vec{p})) = \{1, 4, 5\} = \{4\} \cup \{1,5\} = \NDes(R) \cup d(I,R)$, as desired.
\end{exam}

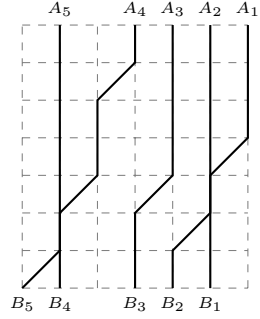
\begin{figure}
    \centering
    \begin{tikzpicture}[scale=.5,baseline={(0,1)}]
    \draw[dashed,gray] (0,1) grid (6, 8);
    \node[above] at (6, 8) {\tiny $A_1$};
    \node[above] at (5, 8) {\tiny $A_2$};
    \node[above] at (4, 8) {\tiny $A_3$};
    \node[above] at (3, 8) {\tiny $A_4$};
    \node[above] at (1, 8) {\tiny $A_5$};
    \node[below] at (5, 1) {\tiny $B_1$};
    \node[below] at (4, 1) {\tiny $B_2$};
    \node[below] at (3, 1) {\tiny $B_3$};
    \node[below] at (1, 1) {\tiny $B_4$};
    \node[below] at (0, 1) {\tiny $B_5$};
    \draw[thick] (6,8) -- (6,5) -- (5,4) -- (5,1);
    \draw[thick] (5,8) -- (5,3) -- (4,2) -- (4,1);
    \draw[thick] (4,8) -- (4,4) -- (3,3) -- (3,1);
    \draw[thick] (3,8) -- (3,7) -- (2,6) -- (2,4) -- (1,3) -- (1,1);
    \draw[thick] (1,8) -- (1,2) -- (0,1) -- (0,1);
    \end{tikzpicture}
    \caption{A subnetwork for the ribbon $R= (2,2,2,2,1)/(1,1,1).$}
    \label{fig:subnetwork for a ribbon}
\end{figure}

Let \(f^{\lambda/\mu}(A)\) denote the number of standard Young tableaux
of shape \(\lambda/\mu\) with descent set \(A\).

\begin{thm}\label{thm: explicit ribbon expansion in full generality}
  Let $R^{(1)}, \ldots, R^{(k)}$ be ribbons, where each
  \( R^{(i)} = \lambda^{(i)}/\mu^{(i)} \) satisfies
  \( \ell(\lambda^{(i)})\le n \). Then for any $I \subseteq [n-1]$, we
  have
\begin{align*}
    &\imm_{\T{I}} \left( \tJT_{R^{(1)}}( \vec{x}^{(1)})
    *\cdots * 
    \tJT_{R^{(k)}}(\vec{x}^{(k)}) \right) \\
    &= \sum_{\nu_1 \vdash m_1, \, \ldots, \, \nu_k \vdash m_k}
    \left(
    \sum_{
    \substack{
    I_i \subseteq [a_i,b_i-1] \\
    I_1 \cap \cdots \cap I_k = I
    }} \,
    \prod_{i=1}^k
    f^{\nu_i} \left( \NDes(R^{(i)}) \cup d(I_i, R^{(i)}) \right)
    \right)
    s_{\nu_1}(\vec{x}^{(1)})
    \cdots 
    s_{\nu_k}(\vec{x}^{(k)}),
\end{align*}
where $m_i$ is the size of $R^{(i)}$ and $[a_i, b_i-1] = \{j:\lambda^{(i)}_j>\mu^{(i)}_j\}$.
\end{thm}

\begin{proof}
Following the proof of \Cref{cor:all TL imms s positive}, we have
\begin{align*}
\imm_{\T{I}} 
\left( \tJT_{R^{(1)}}( \vec{x}^{(1)})
*\cdots * 
\tJT_{R^{(k)}}(\vec{x}^{(k)} ) \right)
&= \sum_{\substack{I_1, \ldots, I_k \subseteq [n-1] \\ I_1 \cap \cdots \cap I_k = I}} \phantom{,} 
\prod_{i=1}^k
\imm_{\T{I_i}} \left( \tJT_{R^{(i)}} ( \vec{x}^{(i)}  ) \right) \\
&= \sum_{\substack{I_1, \ldots, I_k \subseteq [n-1] \\ I_1 \cap \cdots \cap I_k = I}} \phantom{,} 
\prod_{i=1}^k
 \sum_{H^{(i)}} \wt_{\vec{x}^{(i)}} ( H^{(i)} ) ,
\end{align*}
where the last sum above is over the \(R^{(i)}\)-subnetworks
\( H^{(i)} \) satisfying \(\beta(H^{(i)}) = \B{I_i}\). 
Note
that \(\epsilon(H^{(i)}) = 0\) for all \(i\) by \Cref{lem: epsilon for ribbons}.

Observe that the only nontrivial terms of the outer sum occur when
\(I_i \subseteq [a_i, b_i-1]\) for each \( i \). To see this, suppose
\(j \not\in [a_i, b_i-1]\) for some \( j\in I_i \). If \(\vec{p} = (p_1, \ldots, p_n)\) is a
\((\lambda^{(i)}, \mu^{(i)})\)-path family covering \(H^{(i)}\), the
path \(p_j\) is a vertical path because
\(\lambda^{(i)}_j = \mu^{(i)}_j\). The requirement that
\(\beta(H^{(i)}) = \B{I_i}\) would mean that \(p_j\) intersects
\(p_{j+1}\), which is impossible in this situation. Thus we may
rewrite the sum above as
\begin{equation*}
\sum_{
\substack{
I_i \subseteq [a_i,b_i-1] \\
I_1 \cap \cdots \cap I_k = I
}} \phantom{,} 
\prod_{i=1}^k
 \sum_{H^{(i)}} \wt_{\vec{x}^{(i)}} ( H^{(i)} ).
\end{equation*}
Applying \Cref{thm:subnetworks-to-words}, the above is equal to 
\begin{equation}\label{eq: before plugging in ssyt}
\sum_{
\substack{
I_i \subseteq [a_i,b_i-1] \\
I_1 \cap \cdots \cap I_k = I
}} \phantom{,} 
\prod_{i=1}^k
\sum_{ \substack{u^{(i)} \in \ZZ^{m_i}_{\geq 1}\\
\Des(u^{(i)}) = \NDes(R^{(i)}) \cup d(I_i, R^{(i)})
}} (\vec{x}^{(i)})_{u^{(i)}} . 
\end{equation}

It is well-known \cite[7.23.1~Lemma]{EC2} that if \( u \) corresponds to \( (P,Q) \) under Robinson--Schensted--Knuth (RSK) correspondence, then \( \Des(u) = \Des(Q) \). Thus, for each \(i \in [k]\), we have
\begin{align*}
\sum_{ \substack{
u^{(i)} \in \ZZ^{m_i}_{\geq 1} \\
\Des(u^{(i)}) = \NDes(R^{(i)}) \cup d(I_i, R^{(i)})
}} (\vec{x}^{(i)})_{u^{(i)}}
&=
\sum_{\nu_i \vdash m_i}
\sum_{\substack{Q \in \SYT(\nu_i) \\ \Des(Q) = \NDes(R^{(i)}) \cup d(I_i, R^{(i)})}}
 \sum_{P \in \SSYT(\nu_i)} (\vec{x}^{(i)})_{P}  \\
&= 
\sum_{\nu_i \vdash m_i}
f^{\nu_i} \left( \NDes(R^{(i)}) \cup d(I_i, R^{(i)}) \right)
s_{\nu_i}(\vec{x}^{(i)}).
\end{align*}
Substituting the above into \eqref{eq: before plugging in ssyt}, the immanant is equal to
\begin{align*}
&\sum_{
\substack{
I_i \subseteq [a_i,b_i-1] \\
I_1 \cap \cdots \cap I_k = I
}} \phantom{,} 
\prod_{i=1}^k
\left(
\sum_{\nu_i \vdash m_i}
f^{\nu_i} \left( \NDes(R^{(i)}) \cup d(I_i, R^{(i)}) \right)
s_{\nu_i}(\vec{x}^{(i)})
\right) \\
&= \sum_{\nu_1 \vdash m_1, \, \ldots, \, \nu_k \vdash m_k}
\left(
\sum_{
\substack{
I_i \subseteq [a_i,b_i-1] \\
I_1 \cap \cdots \cap I_k = I
}} \phantom{,} 
\prod_{i=1}^k
f^{\nu_i} \left( \NDes(R^{(i)}) \cup d(I_i, R^{(i)}) \right)
\right)
s_{\nu_1}(\vec{x}^{(1)})
\cdots 
s_{\nu_k}(\vec{x}^{(k)}),
\end{align*}
which completes the proof.
\end{proof}

As a conclusion to this subsection, we present some
corollaries of \Cref{thm: explicit ribbon expansion in full
  generality} that concern the case $R^{(1)} = \cdots = R^{(k)}$. 

\begin{cor}\label{cor:ribbon power tableau expansion}
  Let \( R=\lambda/\mu \) be a ribbon of size \( m \) with \( \ell(\lambda) \leq n \),
  and let \( I\subseteq [n-1] \). Then
\begin{align*}
&\imm_{\T{I}} \left( \tJT_{R}(\vec{x}^{(1)}) *\cdots * \tJT_{R}(\vec{x}^{(k)}) \right) \\
&= \sum_{\nu_1,\ldots, \nu_k \vdash m}
\left(
\sum_{\substack{A_1, \ldots, A_k \subseteq [m-1] \\
A_1 \cap \cdots \cap A_k = \NDes(R) \cup d(I,R)}}
f^{\nu_1} \left(A_1\right)
\cdots
f^{\nu_k} \left(A_k\right)
\right)
s_{\nu_1}(\vec{x}^{(1)})
\cdots 
s_{\nu_k}(\vec{x}^{(k)}).
\end{align*}
\end{cor}

\begin{proof}
  Let \( [a,b]=\{i:\lambda_i>\mu_i\} \). Then by \Cref{thm: explicit
    ribbon expansion in full generality}, the immanant is equal
  to
\begin{equation}\label{eq:1}
\sum_{\nu_1 \vdash m_1, \, \ldots, \, \nu_k \vdash m_k}
\left(
\sum_{
\substack{
I_i \subseteq [a,b-1] \\
I_1 \cap \cdots \cap I_k = I
}} \,
\prod_{i=1}^k
f^{\nu_i} \left( \NDes(R) \cup d(I_i, R) \right)
\right)
s_{\nu_1}(\vec{x}^{(1)})
\cdots 
s_{\nu_k}(\vec{x}^{(k)}).
\end{equation}

The map \(c \mapsto \sum_{j=1}^c (\lambda_j-\mu_j)\)
is a bijection from \([a,b-1]\) to \(\Des(R)\).
Thus, as \(I_i\) ranges over all possible subsets of \([a, b-1]\),
the set \(d(I_i,R)\) can be any subset of \(\Des(R)\).
This means that \(\NDes(R) \cup d(I_i,R)\)
can be any subset of \([m-1]\) containing \(\NDes(R)\).
In addition, if
\(I_1 \cap \cdots \cap I_k = I\),
then by the definition of \( d(I,R) \) in \eqref{eq:d(I,R)},
we have
\(d(I_1,R) \cap \cdots \cap d(I_k,R) = d(I,R) \).
It follows that
\[
\bigcap_{i=1}^k \left( \NDes(R) \cup d(I_i, R) \right) 
= \NDes(R) \cup \left( \bigcap_{i=1}^k d(I_i, R)\right) 
= \NDes(R) \cup d(I,R).
\]
Thus, setting \(A_i = \NDes(R) \cup d(I_i,R)\), 
we can rewrite \eqref{eq:1} as the right-hand side
of the formula in the corollary.
\end{proof}

We can now obtain \Cref{thm: Schur expansion of JT_R*JT_R} from the
introduction as a consequence of \Cref{cor:ribbon power tableau
  expansion}. In fact, we can make a more general statement involving
\(k\) ribbons, rather than just two. Let \(\mathcal{R}_m\) denote the
set of ribbons \(R = \lambda/\mu\) of size \(m\) that satisfy
\(\{i:\lambda_i > \mu_i\} = [\ell(\lambda)]\). This condition ensures
that \(\mathcal{R}_m\) does not contain vertical or horizontal
translations of the same Young diagram. The map \(R \mapsto \Des(R)\)
defines a bijection between \(\mathcal{R}_m\) and the set
of all subsets of \([m-1]\). In
other words, up to translation, ribbons are uniquely determined by
their descent sets. We denote \( R' = \lambda'/\mu' \), the conjugate
of \( R=\lambda/\mu \).

\begin{cor}\label{cor:det power expansion}
For any ribbon $R$, we have
\begin{align*}
&\det\left(
\tJT_R(\vec{x}^{(1)})
* \cdots *
\tJT_R(\vec{x}^{(k)})\right) \\
&=
\sum_{\nu_1,\ldots, \nu_k \vdash m}
\left(
\sum_{\substack{ A_1, \ldots, A_k \subseteq [m-1] \\
A_1 \cap \cdots \cap A_k = \Des(R')}}
f^{\nu_1} \left(A_1\right)
\cdots
f^{\nu_k} \left(A_k\right)
\right)
s_{\nu_1}(\vec{x}^{(1)})
\cdots 
s_{\nu_k}(\vec{x}^{(k)}).
\end{align*}
\end{cor}

\begin{proof}
Without loss of generality, we assume \(R \in \mathcal{R}_m\).
Let \(\overline{R}\) be the unique 
ribbon in \(\mathcal{R}_m\) that satisfies
\[\Des(\overline{R}) = [m-1] \setminus \Des(R) = \NDes(R).\]
Note that \(\overline{R}\) is a \(180^\circ\) rotation of \(R'\).
Thus, by \cite[Exercise 7.56.(a)]{EC2},
we have \(s_{\overline{R}}(\vec x) = s_{R'}(\vec x)\).
Generalizing an idea of Stanley \cite{Stanley2024}, we consider the homomorphism 
\(\Phi: \Lambda \to \Lambda(\vec x^{(1)}) \otimes \cdots \otimes \Lambda(\vec x^{(k)})\)
defined by \(\Phi(e_n)=e_n(\vec x^{(1)}) \cdots e_n(\vec x^{(k)})\).
By the Jacobi--Trudi identity, we have
\[\det\left(
\tJT_R(\vec{x}^{(1)})
* \cdots *
\tJT_R(\vec{x}^{(k)})\right) = \Phi(s_{R'}). \]
Since \(s_{R'}(\vec x) = s_{\overline{R}}(\vec x)\), we also have
\[\Phi(s_{R'}) = \Phi(s_{\overline{R}}) = \det\left(
\tJT_{(\overline{R})'}(\vec{x}^{(1)})
* \cdots *
\tJT_{(\overline{R})'}(\vec{x}^{(k)})\right).\]
Recall that \(\det = \imm_{\T{\emptyset}}\). 
By \Cref{cor:ribbon power tableau expansion} with \( I=\emptyset \), we have
\begin{align*}
&\det\left(
\tJT_{(\overline{R})'}(\vec{x}^{(1)})
* \cdots *
\tJT_{(\overline{R})'}(\vec{x}^{(k)})\right) \nonumber \\
&=
\sum_{\nu_1,\ldots, \nu_k \vdash m}
\left(
\sum_{\substack{ A_1, \ldots, A_k \subseteq [m-1] \\
A_1 \cap \cdots \cap A_k = \NDes({(\overline{R})'})}}
f^{\nu_1} \left(A_1\right)
\cdots
f^{\nu_k} \left(A_k\right)
\right)
s_{\nu_1}(\vec{x}^{(1)})
\cdots 
s_{\nu_k}(\vec{x}^{(k)}).
\end{align*}
To conclude, it suffices to observe that
\((\overline{R})' = \overline{R'}\), which implies
\[ \NDes({(\overline{R})'}) = \NDes(\overline{R'}) = \Des(R'). \qedhere \]
\end{proof}

Since \(R \mapsto \Des(R)\) is a bijection between \(\mathcal{R}_m\)
and the set of all subsets of \([m-1]\), summing the equation in
\Cref{cor:det power expansion} over all \( R\in \mathcal{R}_m \) gives
the following proposition.

\begin{prop}\label{cor:all1s-coefficient}
We have
  \[
  \sum_{R \in \mathcal{R}_m} 
  \det\left(
  \tJT_R(\vec{x}^{(1)})
  * \cdots *
  \tJT_R(\vec{x}^{(k)})
  \right)
  = \left(\sum_{\lambda \vdash m} f^{\lambda}s_\lambda(\vec{x}^{(1)}) \right)
  \cdots
  \left(\sum_{\lambda \vdash m} f^{\lambda}s_\lambda(\vec{x}^{(k)}) \right).
  \]
In particular, the coefficient of 
$(x^{(1)}_1 \cdots x^{(1)}_m)
\cdots 
(x^{(k)}_1 \cdots x^{(k)}_m)$
on either side is $(m!)^k$.
\end{prop}

Suppose \(R = \lambda/\mu \in \mathcal{R}_m\) satisfies
\(\ell(\lambda) = n\). For \(I \subseteq [n-1]\), define \(R_I\) to be
the unique ribbon in $\mathcal{R}_m$ satisfying
$\Des\left(R_I\right) = d([n-1] \setminus I,R)$.

\begin{thm}\label{thm:TL of hadamard power equals det}
For any \(R = \lambda/\mu \in \mathcal{R}_m\) with \(\ell(\lambda) \le n\) and any \(I \subseteq [n-1]\), we have
\[\imm_{\T{I}} \left( \tJT_{R}(\vec{x}^{(1)}) *\cdots * \tJT_{R}(\vec{x}^{(k)}) \right) = 
\det \left( \tJT_{R_I} (\vec{x}^{(1)}) *\cdots * \tJT_{R_I}(\vec{x}^{(k)}) \right).\]
\end{thm}

\begin{proof}
  We expand each side of the equation in the theorem via \eqref{eq:
    before plugging in ssyt} from the proof of \Cref{thm: explicit
    ribbon expansion in full generality}. As in the proof of
  \Cref{cor:ribbon power tableau expansion}, we make use here of the
  fact that the condition \(I_1 \cap \cdots \cap I_k = I\) is
  equivalent to
  \(\Des(u^{(1)}) \cap \cdots \cap \Des(u^{(k)}) = \NDes(R) \cup d(I,R)\).
  Thus the left hand side of the equation becomes
    \begin{align*}
    \imm_{\T{I}} \left( \tJT_{R}( \vec{x}^{(1)} ) *\cdots * \tJT_{R}( \vec{x}^{(k)} ) \right)
    &= \sum_{\substack{u^{(1)}, \ldots, u^{(k)} \in \ZZ^m_{\geq 1} \\ \Des(u^{(1)}) \cap \cdots \cap \Des(u^{(k)}) = \NDes(R) \cup d(I,R) } } 
    (\vec{x}^{(1)})_{u^{(1)}} \cdots (\vec{x}^{(k)})_{u^{(k)}}.
    \end{align*}
    Since \(\det = \imm_{\T{\emptyset}}\), we expand the right hand
    side of the equation similarly:
    \begin{align*}
    \det \left( \tJT_{R_I}(\vec{x}^{(1)}) *\cdots * \tJT_{R_I}(\vec{x}^{(k)})\right)
    &= \imm_{\T{\emptyset}} \left( \tJT_{R_I}(\vec{x}^{(1)}) *\cdots * \tJT_{R_I}(\vec{x}^{(k)})\right) \\
    &= \sum_{\substack{u^{(1)}, \ldots, u^{(k)} \in \ZZ^m_{\geq 1} \\ \Des(u^{(1)}) \cap \cdots \cap \Des(u^{(k)}) = \NDes(R_I) }} (\vec{x}^{(1)})_{u^{(1)}} \cdots (\vec{x}^{(k)})_{u^{(k)}}.
    \end{align*}
To finish the proof, it suffices to observe that 
\begin{align*}
\NDes(R_I) &= [m-1] \setminus d([n-1] \setminus I, R) \\
&= [m-1] \setminus \left\{\sum_{j=1}^i (\lambda_j - \mu_j) : i \in [n-1] \setminus I \right\} \\
&= \left\{a \in [m-1] : a \neq  \sum_{j=1}^i (\lambda_j - \mu_j) \mbox{ for all \( i \)} \right\} \cup \left\{\sum_{j=1}^i (\lambda_j - \mu_j) : i \in I \right\} \\
&= \NDes(R) \cup d(I,R). \qedhere
\end{align*}
\end{proof}

\begin{cor}\label{cor: single imm is det}
  For any ribbon \(R = \lambda/\mu\) with \(\ell(\lambda) \le n\) and any \(I \subseteq [n-1]\), we have
    \[ \imm_{\T{I}} \left( \tJT_{R} ({\vec{x}}) \right) = \det \left(\tJT_{R_I}({\vec{x}})\right) = s_{\left(R_I\right)'}(\vec{x}). \]
\end{cor}

\begin{proof}
This is the case $k=1$ of Theorem \ref{thm:TL of hadamard power equals det}.
\end{proof}

Let \((1^m)=(m)'\) denote the partition with \(m\) copies of \(1\).

\begin{cor}
  Every skew Schur function \( s_R(\vec x) \) for a ribbon
  \(R \in \mathcal{R}_m\) is equal to some Temperley--Lieb immanant of
  \(\tJT_{(1^m)}(\vec{x})\). Conversely, any nonzero
  \( \imm_\tau(\tJT_{(1^m)}(\vec{x})) \) is equal to \( s_R(\vec x) \)
  for some \( R \in \mathcal{R}_m \).
\end{cor}

\begin{proof}
Observe that \(d(I,(1^m)) = I\) for any \(I \subseteq [m-1]\). 
Thus \((1^m)_I\) is the unique ribbon in 
\(\mathcal{R}_m\) with descent set \([m-1] \setminus I\).
In particular, setting \(I = \Des(R)\),
we have \((1^m)_{\Des(R)}= \overline{R}\).

Applying \Cref{cor: single imm is det}, we have
\[\imm_{\tau(\Des(R))} \left( \tJT_{(1^m)} ({\vec{x}}) \right)
= s_{((1^m)_{\Des(R)})'}(\vec{x}) 
= s_{(\overline{R})'}(\vec x)
= s_{R}(\vec x).\]
As in the proof of \Cref{cor:det power expansion}, the final equality above follows from \cite[Exercise 7.56.(a)]{EC2}
and the observation that 
\(\overline{R}\) is a \(180^\circ\) rotation of \(R'\).

The second statement follows from \Cref{thm: TL multinetwork expansion}
and \Cref{cor: single imm is det}.
\end{proof}

\subsection{Steps to the full conjecture}\label{sec: limitations}

In view of \Cref{conj: stronger sokal's conjecture}, in this subsection we discuss the extent to which
the proof of monomial-positivity in \Cref{thm: TL multinetwork expansion} can be generalized.
While we can generalize \Cref{thm: TL multinetwork expansion} slightly,
we show that $3 \times 2$ blocks will essentially always generate an obstruction to the proof method.

Call a collection $\lambda^{(1)}/\mu^{(1)}, \ldots, \lambda^{(k)}/\mu^{(k)}$ of skew shapes
\textit{essentially $(3 \times 2)$-avoiding} if 
whenever some $\lambda^{(t)}/\mu^{(t)}$
contains a $3 \times 2$ block spanning rows $a$, $a+1$, and $a+2$,
then there exists $1 \leq s \leq k$ and $b \in \{a,a+1,a+2\}$ such that
$\lambda^{(s)}_b = \mu^{(s)}_b $.

\begin{prop}\label{cor:essential-implies-s-pos}
If $\lambda^{(1)}/\mu^{(1)}, \ldots, \lambda^{(k)}/\mu^{(k)}$
is an essentially $(3 \times 2)$-avoiding collection of skew shapes
with \( \ell(\lambda^{(i)})\le n \) for all \( i \),
then
\(\imm_{\tau} ( \tJT_{\lambda^{(1)} /
\mu^{(1)}}(\vec{x}^{(1)}) * \cdots * \tJT_{\lambda^{(k)} /
\mu^{(k)}}(\vec{x}^{(k)}) )\)
is $s$-positive for any $\tau \in \K_n$.
\end{prop}

\begin{proof}
  By \Cref{thm: TL multinetwork expansion}, we may assume that some
  \(\lambda^{(t)}/\mu^{(t)}\) contains a \(3 \times 2\) block of
  cells. Suppose that \(\lambda^{(t)}/\mu^{(t)}\) contains a
  \(3 \times 2\) block spanning rows \(a\), \(a+1\), and \(a+2\).
  Since this collection is essentially \((3\times 2)\)-avoiding, there
  are integers \(b \in \{a,a+1,a+2\}\) and \( s\in[k] \) such
  that \(\lambda^{(s)}_b = \mu^{(s)}_b \).

  Now we construct another collection of skew shapes
  \(\rho^{(1)}/\nu^{(1)}, \ldots, \rho^{(k)}/\nu^{(k)}\). For each
  \(1 \leq i \leq k\), define the partitions
  \(\rho^{(i)} = (\rho^{(i)}_1,\dots,\rho^{(i)}_n)\) and
  \(\nu^{(i)} = (\nu^{(i)}_1,\dots,\nu^{(i)}_n)\) as follows: for
  \( q\in [n] \),
\[
\rho^{(i)}_q =
\begin{cases}
\lambda^{(i)}_q + \lambda^{(i)}_b  & \text{if } q < b, \\
\lambda^{(i)}_q & \text{otherwise},
\end{cases}
\qquad 
\nu^{(i)}_q =
\begin{cases}
\mu^{(i)}_q + \lambda^{(i)}_b  & \text{if } q < b, \\
\lambda^{(i)}_b  & \text{if } q = b, \\
\mu^{(i)}_q & \text{otherwise}.
\end{cases}
\]
In other words, \( \rho^{(i)} / \nu^{(i)} \) is the skew shape
obtained from \( \lambda^{(i)}/\mu^{(i)} \) by translating the first
\( b-1 \) rows to the right by \( \lambda_b^{(i)} \) units, and
deleting the cells in row \( b \), if any. Note that, by construction,
we have \(\rho^{(i)}_b = \nu^{(i)}_b \) for all \(i\). In particular,
since \(b \in \{a, a+1, a+2\}\), the skew shape
\(\rho^{(t)} / \nu^{(t)}\) does not contain a \(3 \times 2\) block
spanning rows \(a\), \(a+1\), and \(a+2\). Note also that the
collection \(\rho^{(1)}/\nu^{(1)}, \ldots, \rho^{(k)}/\nu^{(k)}\) is
still essentially \((3 \times 2)\)-avoiding, and that this
construction does not create any new \( 3\times 2 \) blocks, thereby
decreasing the number of \( 3\times 2 \) blocks.

For simplicity, set
\[
  M = \tJT_{\lambda^{(1)} / \mu^{(1)}}(\vec{x}^{(1)}) * \cdots *
  \tJT_{\lambda^{(k)} / \mu^{(k)}}(\vec{x}^{(k)})
\]
 and
\[
  N = \tJT_{\rho^{(1)} / \nu^{(1)}}(\vec{x}^{(1)}) * \cdots *
  \tJT_{\rho^{(k)} / \nu^{(k)}}(\vec{x}^{(k)}).
\]
We claim that \(\imm_\tau(M)\) is a Schur-positive multiple of
\(\imm_\tau(N)\) for any \(\tau \in \K_n\). We first show that it
suffices to prove the claim. If some \(\rho^{(i)} / \nu^{(i)}\) still
contains a \(3 \times 2\) block, we repeat the construction from the
previous paragraph. Since the construction decreases the number of
\( 3\times 2 \) blocks, we eventually obtain a collection of
\((3 \times 2)\)-avoiding skew shapes. Then \(\imm_f(M)\) is
Schur-positive by the claim and \Cref{thm: TL multinetwork expansion}.

It remains to prove the claim. The condition
\(\lambda^{(s)}_b = \mu^{(s)}_b \) implies that the \(b\)th diagonal
entry of \(\tJT_{\lambda^{(s)} / \mu^{(s)}}(\vec{x}^{(s)})\) is
\(e_0(\vec x^{(s)}) = 1\). This also implies that
\(\tJT_{\lambda^{(s)} / \mu^{(s)}}(\vec{x}^{(s)})_{i,j} = 0\) for
\(i<b\) or \(j<b\). It follows that there is a block decomposition
\begin{equation}\label{eq:6}
M =
\left[
\begin{NiceArray}{ccc|c|ccc}[margin=2pt]
\Block{3-3}<\Large>{A} & & & * & \Block{3-3}<\Large>{0} & & \\
& & & \vdots & & & \\
& & & * & & & \\
\hline
0 & \cdots & 0 & \varepsilon & * & \cdots & * \\
\hline
\Block{3-3}<\Large>{0} & & & 0 & \Block{3-3}<\Large>{B} & & \\
& & & \vdots & & & \\
& & & 0 & & &
\end{NiceArray}
\right].
\end{equation}
Here \(A\) is a \((b-1) \times (b-1)\) block, 
\(B\) is an \((n-b) \times (n-b)\) block,
and \[\varepsilon = \prod_{j \neq s} e_{\lambda^{(j)}_b - \mu^{(j)}_b}(\vec{x}^{(j)}).\]
Note, in particular, that \(\varepsilon\) is Schur-positive.

Suppose \(\sigma \in \sn\) satisfies
\(M_{1,\sigma(1)} \cdots M_{n,\sigma(n)} \neq 0\).
The block decomposition of \(M\) implies that \(\sigma = \sigma_1\sigma_2\),
where \(\sigma_1\) belongs to the subgroup \(\mathfrak{S}_{[1,b-1]} \subseteq \sn\)
that only permutes the set \([1,b-1]\).
Similarly, \(\sigma_2\) belongs to \(\mathfrak{S}_{[b+1,n]}\),
the subgroup permuting the set \([b+1,n]\).
In particular, \(\sigma\) has a reduced word
using only the generators \(\s_1, \ldots, \s_{b-2}, \s_{b+1}, \ldots, \s_{n-1}\).

As a result, \(\theta(\sigma)\) can be written as a linear of combination
of Kauffman diagrams of the form \(\tau_1\tau_2\), where
\(\tau_1\) is written in the generators \(t_1, \ldots, t_{b-2}\)
and \(\tau_2\) is written in the generators \(t_{b+1}, \ldots, t_{n-1}\).
It follows that \(\imm_\tau(M)=0\) unless \(\tau\) is of this form.
Note that, in this situation, \(f_{\tau_1\tau_2}(\sigma_1\sigma_2) = f_{\tau_1}(\sigma_1)f_{\tau_2}(\sigma_2)\).
Thus we have 
\begin{equation}\label{eq:8}
  \imm_{\tau_1\tau_2}(M) = \varepsilon \imm_{\tau_1}(A) \imm_{\tau_2}(B).
\end{equation}

Now observe that, for each \(m\in [k]\) and \( i,j\in [n] \) with \(j \ne b \),
we have 
\[
  \rho^{(m)}_i - \nu^{(m)}_j - i + j = \lambda^{(m)}_i - \mu^{(m)}_j - i + j,
\]
which holds for \( j>b \) by construction, and also for \( j<b \),
since, in this case,
\[
  \rho^{(m)}_i - \nu^{(m)}_j - i + j = (\lambda^{(m)}_i +
  \lambda^{(m)}_b) - (\mu^{(m)}_j + \lambda^{(m)}_b) - i + j =
  \lambda^{(m)}_i - \mu^{(m)}_j - i + j.
\]
Thus the matrices \(M\) and \( N \) are identical except in column \(b\).
Recall that \(\rho^{(m)}_b = \nu^{(m)}_b \) for all \(m\).
Thus \(N\) has the same block decomposition as \( M \) in \eqref{eq:6},
except with a \(1\) in place of \(\varepsilon\).
Thus \(\imm_{\tau_1\tau_2}(N) = \imm_{\tau_1}(A) \imm_{\tau_2}(B)\).
Therefore, by \eqref{eq:8}, \(\imm_{\tau_1\tau_2}(M)\) is a Schur-positive multiple of \(\imm_{\tau_1\tau_2}(N)\).
This proves the claim and completes the proof.
\end{proof}

The proof of \Cref{thm: TL multinetwork expansion} had two primary steps.
First, we expanded each Temperley--Lieb immanant as a sum over subnetworks, i.e., wrote
\[\imm_{\tau} \left( \tJT_{\lambda^{(1)} / \mu^{(1)}}(\vec{x}^{(1)}) *\cdots * \tJT_{\lambda^{(k)} / \mu^{(k)}}(\vec{x}^{(k)}) \right) = \sum_{\vec{H}} f_\tau(\beta(\vec{H})) \wt_{\underline{\vec x}}(\vec{H}).\]
Then, we argued that $f_\tau(\beta(\vec{H})) \geq 0$ whenever $\vec{H}$ is a multinetwork for a collection of $(3 \times 2)$-avoiding skew shapes.
The result below demonstrates that the proof method for \Cref{thm: TL multinetwork expansion}
will not work for collections of skew shapes that are not essentially $3 \times 2$ avoiding.
Note that if \( \tau=1 \), then \( f_\tau=f_1 \) is the sign character \( \sgn \).

\begin{prop}\label{prop: counterexample}
  Let $\lambda^{(1)} / \mu^{(1)}, \ldots, \lambda^{(k)} / \mu^{(k)}$
  be a collection of connected skew shapes
  that is not essentially $(3 \times 2)$-avoiding.
  Then there exists a multinetwork $\vec{H}$ for this collection satisfying $ f_1(\beta(\vec{H})) = \sgn(\beta(\vec{H})) < 0$.
\end{prop}

\begin{proof}
  As in the proof of \Cref{lem: beta for multinetworks}, let \(\star\)
  be the product in \(\CC[\sn]\) defined by
  \(u \star w = u\delta_{u,w}\) for \(u,w \in \sn\) and extended by
  linearity. Let \( n \) be an integer with
  \( \ell(\lambda^{(i)})\le n \) for each \( i \). By the assumption,
  there is \(t\in [k]\) such that \(\lambda^{(t)} / \mu^{(t)}\)
  contains a \(3 \times 2\) block of cells spanning rows \(a\),
  \(a+1\), and \(a+2\), and that \(\lambda^{(s)}_b > \mu^{(s)}_b \)
  for all \(s \neq t\) and all \(b \in \{a,a+1,a+2\}\). It follows
  that $\lambda^{(t)}_{a+2}-2 \geq \mu^{(t)}_{a}$. In particular,
  $A(\lambda^{(t)})_{a+2}$ lies weakly to the right of
  $B(\mu^{(t)})_a$ in the plane. Therefore, we may construct a
  $(\lambda^{(t)}, \mu^{(t)})$-path family
  $\vec{p} = (p_1, \ldots, p_n)$ with
  \begin{align*}
  p_{a+2} &: A(\lambda^{(t)})_{a+2} \to B(\mu^{(t)})_a, & & &
  p_{a+1} &: A(\lambda^{(t)})_{a+1} \to B(\mu^{(t)})_{a+2}, \\
  p_{a} &: A(\lambda^{(t)})_{a} \to B(\mu^{(t)})_{a+1}, & \text{ and } & &
  p_j &: A(\lambda^{(t)})_j \to B(\mu^{(t)})_j \text{ for } j \not\in \{a,a+1,a+2\}.
  \end{align*}
  We may do this in such a way that the
  \((\lambda^{(t)}, \mu^{(t)})\)-subnetwork \(H^{(t)}\) covered
  by \(\vec p\) with intersection sequence \( (a+1,a) \) as shown in 
  the diagram below, where \( A_i = A(\lambda^{(t)})_{i} \)
  and \( B_i = B(\mu^{(t)})_{i} \).
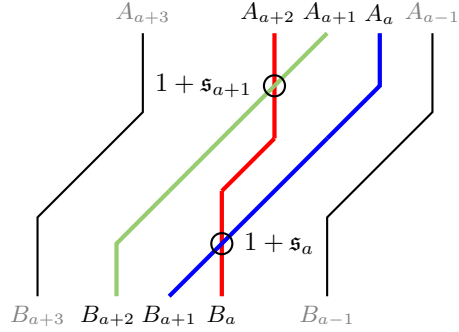
\begin{center}
  \begin{tikzpicture}[scale=.35,baseline={(0,1)}]
    \node[above] at (9, 10) {\small $A_a$};
    \node[above] at (7.125, 10) {\small $A_{a+1}$};
    \node[above] at (4.75, 10) {\small $A_{a+2}$};
    \node[below] at (3, 0) {\small $B_a$};
    \node[below] at (1, 0) {\small $B_{a+1}$};
    \node[below] at (-1.333, 0) {\small $B_{a+2}$};
    \draw[ultra thick, red] (5,10) -- (5,6) -- (5,6);
    \draw[ultra thick, red] (5, 6) -- (3, 4);
    \draw[ultra thick, red] (3, 4) -- (3,4) -- (3,0);
    \draw[ultra thick, YellowGreen] (7,10) -- (5,8) -- (5,8);
    \draw[ultra thick, YellowGreen] (5, 8) -- (1,4);
    \draw[ultra thick, YellowGreen] (1,4) -- (-1,2) -- (-1,0);
    \draw[ultra thick, blue] (9,10) -- (9,8) -- (7,6);
    \draw[ultra thick, blue] (7, 6) -- (3, 2);
    \draw[ultra thick, blue] (3,2)-- (1,0);
    \node[above] at (11, 10) {\small {\color{gray}$A_{a-1}$}};
    \node[above] at (0, 10) {\small {\color{gray}$A_{a+3}$}};
    \node[below] at (7, 0) {\small {\color{gray}$B_{a-1}$}};
    \node[below] at (-4, 0) {\small {\color{gray}$B_{a+3}$}};
    \draw[thick] (11,10) -- (11,7) -- (7,3) -- (7,0);
    \draw[thick] (0,10) -- (0,7) -- (-4,3) -- (-4,0);
    \draw[thick] (5,8) circle (4mm);
    \node[left] at (4.5,8) {$1+\s_{a+1}$};
    \draw[thick] (3,2) circle (4mm);
    \node[right] at (3.5,2) {$1+\s_{a}$};
    \end{tikzpicture}
\end{center}
Then, by \Cref{lem:beta for wiring diagrams}, we have
\(\beta(H^{(t)}) = (1+\s_{a+1})(1+\s_a)\).

  Recall that \(\lambda^{(s)}_b > \mu^{(s)}_b \) for all \(s \neq t\)
  and all \(b \in \{a,a+1,a+2\}\). Since each of these skew shapes is
  connected, this implies that for \( i=a \) and \( i=a+1 \), we have
  \(\mu^{(s)}_i \leq \lambda^{(s)}_{i+1}-1\), that is,
  \(A(\lambda^{(s)})_{i+1}\) lies weakly to the right of
  \(B(\lambda^{(s)})_i\) in the plane. Thus we may construct a
  $(\lambda^{(s)},\mu^{(s)})$-path family $\vec{p}= (p_1, \ldots, p_n)$ so that
        \begin{align*}
        p_{a+2} &: A(\lambda^{(s)})_{a+2} \to B(\mu^{(s)})_{a+1}, & & &
        p_{a+1} &: A(\lambda^{(s)})_{a+1} \to B(\mu^{(s)})_{a}, \\
        p_{a} &: A(\lambda^{(s)})_{a} \to B(\mu^{(s)})_{a+2}, & \text{ and } & &
        p_j &: A(\lambda^{(s)})_j \to B(\mu^{(s)})_j \text{ for } j \not\in \{a,a+1,a+2\}.
        \end{align*}
  We may do this in such a way that the
  \((\lambda^{(s)}, \mu^{(s)})\)-subnetwork \(H^{(s)}\) covered
  by \(\vec p\) with intersection sequence \( (a,a+1) \) as shown in 
  the diagram below, where \( A_i = A(\lambda^{(s)})_{i} \)
    and \( B_i = B(\mu^{(s)})_{i} \).
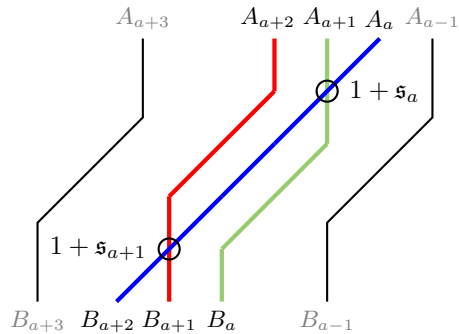
   \begin{center}
     \begin{tikzpicture}[scale=.35,baseline={(0,1)}]
    \node[above] at (9, 10) {\small $A_a$};
    \node[above] at (7.125, 10) {\small $A_{a+1}$};
    \node[above] at (4.75, 10) {\small $A_{a+2}$};
    \node[below] at (3, 0) {\small $B_a$};
    \node[below] at (1, 0) {\small $B_{a+1}$};
    \node[below] at (-1.333, 0) {\small $B_{a+2}$};
    \draw[ultra thick, red] (5,10) -- (5,8);
    \draw[ultra thick, red] (5, 8) -- (1, 4);
    \draw[ultra thick, red] (1,4) -- (1,0);
    \draw[ultra thick, YellowGreen] (7,10) -- (7,6);
    \draw[ultra thick, YellowGreen] (7,6) -- (3,2);
    \draw[ultra thick, YellowGreen] (3,2) -- (3,0);
    \draw[ultra thick, blue] (9,10) -- (5,6);
    \draw[ultra thick, blue] (5, 6) -- (3, 4);
    \draw[ultra thick, blue] (3,4)-- (-1,0);
    \node[above] at (11, 10) {\small {\color{gray}$A_{a-1}$}};
    \node[above] at (0, 10) {\small {\color{gray}$A_{a+3}$}};
    \node[below] at (7, 0) {\small {\color{gray}$B_{a-1}$}};
    \node[below] at (-4, 0) {\small {\color{gray}$B_{a+3}$}};
    \draw[thick] (11,10) -- (11,7) -- (7,3) -- (7,0);
    \draw[thick] (0,10) -- (0,7) -- (-4,3) -- (-4,0);
    \draw[thick] (7,8) circle (4mm);
    \node[right] at (7.5,8) {$1+\s_{a}$};
    \draw[thick] (1,2) circle (4mm);
    \node[left] at (0.5,2) {$1+\s_{a+1}$};
    \end{tikzpicture}
   \end{center}
Then, by \Cref{lem:beta for wiring diagrams}, we have
\(\beta(H^{(t)}) = (1+\s_{a})(1+\s_{a+1})\).
 
    The resulting multinetwork $\vec{H} = (H^{(1)}, \ldots, H^{(k)})$ satisfies
    \begin{align*}
    \beta(\vec{H}) &= \underbrace{\left[ (1+\s_{a})(1+\s_{a+1}) \right] \star \cdots \star \left[ (1+\s_{a})(1+\s_{a+1}) \right]}_{k-1 \rm\ factors} \star \left[ (1+\s_{a+1})(1+\s_{a}) \right] \\
    &= \left[ (1+\s_{a})(1+\s_{a+1}) \right] \star \left[ (1+\s_{a+1})(1+\s_{a}) \right] \\
    &= \left[ 1 + \s_a + \s_{a+1} + \s_a \s_{a+1} \right] \star \left[ 1 + \s_a + \s_{a+1} + \s_{a+1} \s_{a} \right] \\ 
    &= 1 + \s_a + \s_{a+1}.
    \end{align*}
    In particular, \(\sgn(\beta(\vec{H})) = \sgn(1 + \s_a + \s_{a+1}) = 1 -1 -1 = -1. \)
\end{proof}

The inequality $\sgn(\beta(\vec H)) \geq 0$ means that there exists an
injection from the set of path families covering $H$ with negative
sign into the set of those with positive sign. \Cref{prop:
  counterexample} demonstrates that, outside the setting of
\Cref{cor:essential-implies-s-pos}, there is no such injection.
Therefore, if one were to try to solve \Cref{conj: Sokal's conjecture}
or \Cref{conj: stronger sokal's conjecture} by constructing a
sign-reversing injection of path families, the injection must not
always preserve the underlying multiset of edges used.

In addition, the required injection must depend on the directed graph structure specifically for south-southwest lattice paths. 
Consider ordinary Jacobi--Trudi matrices,
\[
  \HJT_{\lambda / \mu}(\vec x) := (h_{\lambda_i-\mu_j + i-j}(\vec
  x))_{i,j=1}^{\ell(\lambda)}.
\]
Note that
$\det(\HJT_{\lambda / \mu}(\vec x)) = s_{\lambda/\mu}(\vec x)$.
Similarly to dual Jacobi--Trudi matrices, ordinary Jacobi--Trudi
matrices can be realized by a lattice path construction with
south $(0, -1)$ and west $(-1, 0)$ steps. However, Hadamard
products of ordinary Jacobi--Trudi matrices do not always have
monomial-positive determinants. One can check, for example, that
$\det(\HJT_{(2,2,2)}(\vec{x}) * \HJT_{(2,2,2)}(\vec{y}))$ is not
monomial-positive.

\section{Representation-theoretic interpretation}\label{Sec: rep theory construction}

To provide further insight on Sokal's conjecture, Stanley \cite{Stanley2024} 
introduced a ring homomorphism 
\( \Phi_{\vec x,\vec y}: \Lambda \to \Lambda(\vec x) \otimes
\Lambda(\vec y) \) defined by
\( \Phi_{\vec x,\vec y}(e_n) = e_n(\vec x) e_n(\vec y) \).
In \Cref{thm: Schur expansion of JT_R*JT_R}, we have seen that
\[
\Phi_{\vec x,\vec y}(s_{R}) 
= \det(\tJT_{R'}(\vec x) * \tJT_{R'}(\vec y)).
\]
is \( s \)-positive for any ribbon $R$.
This naturally leads
to the question of constructing an
\( \mathfrak{S}_n \times \mathfrak{S}_n \)-module whose image
under the Frobenius characteristic map equals $\Phi_{\vec x,\vec y}(s_{R})$.
In this section, we answer this question by constructing a
simplicial complex arising from the Boolean poset whose reduced homology
yields the desired \( \mathfrak{S}_n \times \mathfrak{S}_n \)-module.
To do this, we need some preparations.

From now on, we assume that \(P\) is a finite poset with a unique
minimal element \(\hat{0}\) and a unique maximal element \(\hat{1}\),
and that every maximal chain in \(P\) has the same length, denote by
\(\ell(P)\). Here, the \emph{length} of a chain is the number of
elements in the chain minus one. Let
\(r: P \to \{0,1,\dots,\ell(P)\}\) denote the rank function of \(P\).
Note that $r(\hat{0})=0$ and $r(\hat{1})=\ell(P)$. For
\(A \subseteq [\ell(P)-1]\), the \emph{rank-selected subposet} \(P_A\)
is defined as
\[
    P_A := \{s \in P : r(s) \in A\} \cup\{\hat0,\hat1\}.
\]
The \emph{order complex} \(\Delta(P)\) of \( P \) is the simplicial
complex whose \(k\)-faces are the chains \(x_0 < x_1 < \cdots < x_k\)
in \(P \setminus \{\hat{0}, \hat{1}\}\). We denote by
\(\tilde{H}_i(P) := \tilde{H}_i(\Delta(P))\) the \(i\)-th reduced
simplicial homology group of \( \Delta(P) \) (over \(\mathbb{C}\)).
Recall that for any simplicial complex \( \mathcal{K} \), we have
\( \tilde{H}_{-1}(\mathcal{K}) = 0 \) unless \( \mathcal{K} \) is the
empty simplex \( \emptyset \). For \( \mathcal{K}=\emptyset \), we
have \( \tilde{H}_{-1}(\emptyset) = \mathbb{C} \), and
\( \tilde{H}_i(\emptyset) = 0 \) for all \( i \ge 0 \).

We need two important results in poset topology. One is a result
by Stanley \cite{Stanley1982} on group actions on finite posets, and
the other is a result by Bj\"orner and Wachs \cite{Bjorner1983} on
EL-shellable posets.

Suppose \( G \) is a group of automorphisms of a poset \( P \) with
\( \ell(P)=n \). For a subset \( A \subseteq [n-1] \), the group
\( G \) acts on the set of maximal chains of the rank-selected
subposet \( P_A \). We denote by \( \alpha_A \) the \( G \)-module
induced by this action. Note that \(G\) also acts on each reduced
homology group \(\tilde{H}_i(P_A)\) for \(-1 \leq i < |A|\). Hence, we
also consider \( \tilde{H}_i(P_A) \) as a \( G \)-module. Define the
virtual \( G \)-module \(\beta_A\) by
\begin{equation}\label{eq:5}
    \beta_A := \sum_{i=-1}^{|A|-1} (-1)^{|A| - i - 1} \tilde{H}_i(P_A).
\end{equation}
Using a version of the Hopf--Lefschetz
fixed-point formula, Stanley proved the following theorem.
\begin{thm}\label{thm: Stanley's relations for alpha and beta}
  \cite[1.1.~Theorem]{Stanley1982} Given a group \( G \) of
  automorphisms of a poset \( P \), the \( G \)-modules \(\alpha_A\)
  and the virtual \( G \)-modules \(\beta_A\) are related by
\begin{align}\label{eq: Stanley's relations for alpha and beta}
    \alpha_A = \sum_{B \subseteq A} \beta_B, \qquad \text{and} \qquad 
    \beta_A = \sum_{B \subseteq A} (-1)^{|A| - |B|} \alpha_B.  
\end{align}  
\end{thm}  

Next, we recall the notion of EL-shellable posets. For a poset \( P \)
with \( \ell(P)=n \), let \( E(P) \) denote the set of edges in the
Hasse diagram of \( P \). An \emph{edge labeling} of a poset \( P \)
is a function \( L: E(P) \to K \), where \( K \) is a poset equipped
with a partial order \( <_K \). Given such a labeling, we can
associate a sequence of labels to each maximal chain
\begin{equation}\label{eq:4}
c = \{\hat{0}=u_0 \lessdot u_1 \lessdot \cdots \lessdot u_n \lessdot u_{n+1} = \hat{1}\}
\end{equation}  
by defining its corresponding word as  
\[
L(c) = L(u_0, u_1)L(u_1, u_2) \cdots  L(u_n, u_{n+1}).
\]  
A maximal chain is said to be \emph{increasing} if its labels form a strictly increasing sequence in \( K \), meaning that  
\[
  L(u_0,u_1) <_K L(u_1,u_2) <_K \dots <_K L(u_n,u_{n+1}).
\]  
The \emph{descent set} of the maximal chain \( c \) in \eqref{eq:4} is
defined to be
\[
    \Des(c) := \{ i \in [n] : L(u_{i-1}\lessdot u_i) \ge_K L(u_i \lessdot u_{i+1})\}.
\]
In particular, an increasing maximal chain \(c\) has an empty descent set.

The lexicographic order on these words induces an ordering on the
maximal chains of \( P \). This induced ordering plays a crucial role
in defining a structural property of posets known as EL-shellability.

\begin{defn} \cite[Definition~2.1]{Bjorner1983} A bounded poset
  \( P \) is said to be \emph{edge lexicographically shellable
    (EL-shellable)} if there exists an edge labeling of \( P \) such
  that, within every closed interval \( [x,y] \) of \( P \), there is
  a unique increasing maximal chain. Moreover, this chain must be
  the lexicographically smallest among all maximal chains in \( [x,y] \).
\end{defn}  

EL-shellability has significant implications for the topology of order
complexes. In particular, it allows us to describe the homotopy type
of rank-selected subposets. 

\begin{thm}\label{thm: rank selection of EL-shellable poset}
  \cite[Theorem 8.1]{Bjorner1983} Let \( P \) be a pure, EL-shellable
  poset, and let \( A\subseteq [\ell(P)-1] \). Then the order complex
  of the rank-selected subposet \( P_A \) is homotopy equivalent to a
  wedge of \((|A|-1)\)-dimensional spheres.
\end{thm} 

We now define the main ingredients of the representation-theoretic
interpretation for the expression 
\(\Phi_{\vec x,\vec y}(s_{R})\).
Let \(B_n\) be the
Boolean poset of \([n]\), that is, the poset of all subsets of
\( [n] \) ordered by inclusion. We define another poset
\(\tilde{B}_n\) by
\[
    \tilde{B}_n := \{(S, T) \in B_n \times B_n : |S| = |T|\},
\]  
where \((S, T) \le (A, B)\) in \(\tilde{B}_n\) if
\(S \subseteq A\) and \(T \subseteq B\). There is a natural action of
the symmetric group \(\mathfrak{S}_n\) on \(B_n\), given by
\(\sigma \cdot S = \{\sigma(s) : s \in S\}\). This action extends to
an \(\mathfrak{S}_n \times \mathfrak{S}_n\)-action on \(\tilde{B}_n\),
defined by
\[
    (\sigma, \tau) \cdot (S, T) = (\sigma \cdot S, \tau \cdot T).
\]
Moreover, the group \(\mathfrak{S}_n \times \mathfrak{S}_n\) also acts
on each reduced homology group \(\tilde{H}_{i}((\tilde{B}_n)_A)\) of
the rank-selected subposet \((\tilde{B}_n)_A\). Thus, we can consider
\(\tilde{H}_{i}((\tilde{B}_n)_A)\) as an
\(\mathfrak{S}_n \times \mathfrak{S}_n\)-module. Now we state the main
result of this section.
\begin{thm}\label{thm: representation theoretic model for Phi(s_R)}
  Let \( R \) be a ribbon of size \(n\), and consider the
  \(\mathfrak{S}_n\times \mathfrak{S}_n\)-module
\[
    \tilde{H}_{|\Des(R')|-1}((\tilde{B}_n)_{\Des(R')})
\] 
as described above. Then we have
\[
\Phi_{\vec x, \vec y} (s_R)
=\omega\left(\Frob_{\vec x, \vec y}
\left(\tilde{H}_{|\Des(R')|-1}((\tilde{B}_n)_{\Des(R')})\right)\right).
\]
\end{thm}

Before proving \Cref{thm: representation theoretic model for Phi(s_R)}, we introduce a lemma and recall some relevant terminology. 
A \emph{composition} $\alpha = (\alpha_1, \dots, \alpha_k)$ is a sequence of positive integers. The \emph{size} of a composition is defined as $\alpha_1 + \cdots + \alpha_k$.
We define \( h_\alpha(\vec x) = h_{\alpha_1}(\vec x) \cdots h_{\alpha_k}(\vec x) \)
and \( e_\alpha(\vec x) = e_{\alpha_1}(\vec x) \cdots e_{\alpha_k}(\vec x) \).

Given a subset \( A = \{a_1 < \cdots < a_k\} \subseteq [n-1] \), we define the composition \( \comp(A) \) of size \( n \) by
\[
    \comp(A) = (a_1, a_2 - a_1, \dots, a_k - a_{k-1}, n - a_k).
\]
An \emph{ordered set partition} of \([n]\) is a set partition whose parts are equipped with a total order. For example, $\pi = (\{2,4\}, \{6\}, \{1,3,5\})$ is an ordered set partition of $[n]$ with three parts. The \emph{type} of an ordered set partition is the composition given by the sequence of sizes of its parts. In our running example, the type of $\pi$ is $(2,1,3)$.

\begin{lem}\label{lem:1}
For \( A\subseteq [n-1] \), we have
\[
  \Frob_{\vec x,\vec y}(\tilde{H}_{|A|-1}((\tilde{B}_n)_A)) =
  \sum_{B\subseteq A} (-1)^{|A|-|B|} h_{\comp(B)}(\vec x)h_{\comp(B)}(\vec y).
\]  
\end{lem}

\begin{proof}
Consider the edge labeling \( L : E(\tilde{B}_n) \to \mathbb{Z} \times \mathbb{Z} \) of \( \tilde{B}_n \) defined by  
\[
    L((A,B),(A\cup\{a\},B\cup\{b\})) = (a,b),
\]  
where, we regard \( \mathbb{Z} \times \mathbb{Z} \) as a poset with the order relation  
\[
    (a,b) <_{\mathbb{Z}\times\mathbb{Z}} (c,d) \quad \text{if and only if} \quad a < c \text{ and } b < d.
\]  
It is straightforward to verify that \( L \) is an EL-labeling of
\( \tilde{B}_n \). Consequently, by applying Theorem~\ref{thm: rank
  selection of EL-shellable poset}, we conclude that the reduced
homology \( \tilde{H}_i((\tilde{B}_n)_A) \) vanishes unless
\( i = |A| - 1 \), the highest possible degree. Thus, by \eqref{eq:5},
we obtain the identity
\[
    \beta_A = \sum_{i=-1}^{|A|-1} (-1)^{|A| - i - 1} \tilde{H}_i((\tilde{B}_n)_A) = \tilde{H}_{|A|-1}((\tilde{B}_n)_A).
\]
Hence, by \Cref{thm: Stanley's relations for alpha and beta}, we have
\begin{equation}\label{eq:12}
\tilde{H}_{|A|-1}((\tilde{B}_n)_A) = \sum_{B\subseteq A} (-1)^{|A|-|B|} \alpha_B.
\end{equation}

Let \( A = \{a_1 < \cdots < a_k\} \subseteq [n-1] \), and
consider the rank-selected subposet \( (B_n)_A \) of the Boolean poset
\( B_n \). The maximal chains in \( (B_n)_A \) correspond to the
ordered set partitions of \([n]\) into parts of sizes given by \(\comp(A)\).
The symmetric group \( \mathfrak{S}_n \) acts on these chains by
permuting the underlying elements of \([n]\), and this action realizes
the permutation module 
\( \mathbf{1} \uparrow^{\mathfrak{S}_n}_{\mathfrak{S}_{\comp(A)}}, \)
induced from the trivial module of the Young subgroup
\[
  \mathfrak{S}_{\comp(A)} := 
  \mathfrak{S}_{a_1} \times \mathfrak{S}_{a_2-a_1}\times \cdots \times \mathfrak{S}_{n - a_k}.
\]

Now consider the rank-selected subposet \( (\tilde{B}_n)_A \) of the
poset \( \tilde{B}_n \). The maximal chains in \( (\tilde{B}_n)_A \)
correspond to pairs \( (\pi_1, \pi_2) \) of ordered set partitions of
\([n]\), each of a fixed type \( \comp(A) \). The natural action of
\( \mathfrak{S}_n \times \mathfrak{S}_n \) on such pairs induces the
tensor product
of \( \mathbf{1} \uparrow^{\mathfrak{S}_n}_{\mathfrak{S}_{\comp(A)}} \), that is,
\[
    \alpha_A \cong \left( \mathbf{1} \uparrow^{\mathfrak{S}_n}_{\mathfrak{S}_{\comp(A)}} \right) \otimes \left( \mathbf{1} \uparrow^{\mathfrak{S}_n}_{\mathfrak{S}_{\comp(A)}} \right).
\]
Under the Frobenius characteristic map, the module corresponds to a product of homogeneous
symmetric functions:
\begin{equation}\label{eq:7}
  \Frob_{\vec x,\vec y}(\alpha_A) = h_{\comp(A)}(\vec{x}) h_{\comp(A)}(\vec{y}).
\end{equation}
By
\eqref{eq:7}, applying the Frobenius characteristic map to
\eqref{eq:12} yields the desired identity.
\end{proof}

Finally, we are ready to prove \Cref{thm: representation theoretic model for Phi(s_R)}.

\begin{proof}[Proof of Theorem~\ref{thm: representation theoretic model for Phi(s_R)}]
  By applying the inclusion-exclusion principle as in the proof of
  \cite[Proposition 7.19.1]{EC2}, we obtain that
\begin{equation}\label{eq: s_R' e-expansion}
    s_{R} = \sum_{A\subseteq \Des(R')} (-1)^{|\Des(R')|-|A|} e_{\comp(A)}.
\end{equation}
Therefore, the identity we need to show can be restated as
\[
  \Frob_{\vec x,\vec y}\left(\tilde{H}_{|\Des(R')|-1}((\tilde{B}_n)_{\Des(R')})\right)
  =  \sum_{A\subseteq \Des(R')} (-1)^{|\Des(R')|-|A|} h_{\comp(A)}(\vec x) h_{\comp(A)}(\vec y).
\]
This follows from \Cref{lem:1}.
\end{proof}

\section{Concluding Remarks}\label{Sec: Concluding remarks}

In this section, we review related conjectures and discuss potential
generalizations of our results. 
To introduce the first conjecture, recall Schur’s classical result~\cite{Schur1914}:  
for two real-rooted polynomials 
\[
   f(x) = a_0 + a_1 x + \cdots + a_n x^n 
   \quad \text{and} \quad 
   g(x) = b_0 + b_1 x + \cdots + b_m x^m,
\]
their \emph{factorial Hadamard product}
\[
   f \ostar g := \sum_{i=0}^{n} i! \, a_i b_i \, x^i
\]
is also real-rooted. In analogy with the way Sokal’s conjecture (\Cref{conj: Sokal's conjecture}) generalizes Mal\'{o}’s result, Sokal further proposed the following conjecture, which strengthens Schur’s theorem.

\begin{conj} \cite{Sokal2024} Let
  \( \lambda = (\lambda_1, \dots, \lambda_n) \) and
  \( \mu = (\mu_1, \dots, \mu_n) \) be partitions with at most \( n \)
  parts satisfying \( \mu \subseteq \lambda \). For any integers
  \( k \) and \( r \) with \( 0 \leq r \leq k-1 \), the following
  determinant is \(m\)-positive:
\[
    \det\left( ((\lambda_i - \mu_j - i+ j )!)^r e_{\lambda_i - \mu_j - i + j}(\vec{x}^{(1)}) \cdots e_{\lambda_i - \mu_j  - i + j}(\vec{x}^{(k)}) \right)_{i,j=1}^{n}. \]
\end{conj}

Sokal's conjecture (\Cref{conj: Sokal's conjecture}) can be stated as
$\Phi_{\vec x,\vec y}(s_{\lambda / \mu})$ being $m$-positive; see
\Cref{Sec: rep theory construction} for the definition of the map
\( \Phi_{\vec x,\vec y} \). Stanley extended Sokal's conjecture to a
broader setting.

\begin{conj}\cite{Stanley2024}
  For any partition \( \lambda \), the multi-symmetric function
  \( \Phi_{\vec x,\vec y}(m_\lambda) \) is \( m \)-positive.
\end{conj}  

Recall that Greene~\cite{Greene1992} proved that ordinary immanants of Jacobi--Trudi matrices are monomial-positive. Stembridge~\cite{Stembridge1992a} proposed several conjectural
generalizations of Greene’s theorem. To state his conjectures, we
define the \emph{monomial (virtual) character} \( \phi_\lambda \) as
follows: For \( \lambda\vdash n \) and \( w \in \mathfrak{S}_n \), 
\[
    \phi_\lambda(w) := \langle m_\lambda, p_{\tau(w)} \rangle,
\]  
where \( \langle \cdot ,\cdot \rangle \) is the Hall inner product, and
\( p_{\tau(w)} \) is the power sum symmetric function indexed by the
partition \( \tau(w) \) corresponding to the cycle type of \( w \).
Among Stembridge's conjectures are the following:
\begin{enumerate}
\item Ordinary immanants of Jacobi--Trudi matrices are \(s\)-positive.  
\item Monomial immanants of Jacobi--Trudi matrices are \(m\)-positive.  
\item Monomial immanants of Jacobi--Trudi matrices are \(s\)-positive.  
\end{enumerate}

It is worth noting that the last conjecture implies all the others.
Haiman \cite{Haiman1993} confirmed the first conjecture using
Kazhdan--Lusztig theory. The last conjecture remains an open problem
of significant importance. The positivity of the coefficient of
\( s_{(n)} \) in this case is equivalent to the famous
Stanley--Stembridge conjecture \cite{Stanley1993}, which was recently
resolved by Hikita \cite{Hikita2024}.

Next, we propose a possible extension of Greene’s result and the
second conjecture mentioned above to our setting by considering the
monomial positivity of the immanants of the Hadamard product of
Jacobi--Trudi matrices.

\begin{conj}\label{conj: monomial immanant of JT*JT monomial positive}
  Let \( \lambda\), \( \mu \), \( \nu \), and \( \rho \) be partitions
  of length at most \( n \), satisfying \( \mu \subseteq \lambda \)
  and \( \rho \subseteq \nu \). For any partition \( \eta \vdash n \),
  \begin{enumerate}
  \item the ordinary immanant
    \( \imm_\eta\left( \JT_{\lambda/\mu}(\vec x) * \JT_{\nu/\rho}(\vec
      y) \right) \) is \( m \)-positive,
  \item the monomial immanant
\( \imm_{\phi_\eta}\left( \JT_{\lambda/\mu}(\vec x) *
\JT_{\nu/\rho}(\vec y) \right) \)  
is \( m \)-positive.
  \end{enumerate}
\end{conj}  

Note that the second part of \Cref{conj: monomial immanant of JT*JT
  monomial positive} implies the first. \Cref{conj: Sokal's conjecture} is equivalent to the special case of the first part for
\( \eta=(1^n) \).

\bibliographystyle{abbrv}
\bibliography{local.bib}

\begin{thebibliography}{10}

\bibitem{Bjorner1983}
A.~Bj\"orner and M.~Wachs.
\newblock {On lexicographically shellable posets}.
\newblock {\em Trans. Amer. Math. Soc.}, 277(1):323--341, 1983.

\bibitem{FanGreen}
C.~Fan and R.~Green.
\newblock Monomials and {T}emperley--{L}ieb algebras.
\newblock {\em Journal of Algebra}, 190(2):498--517, 1997.

\bibitem{Fomin2010}
S.~Fomin.
\newblock {Total positivity and cluster algebras}.
\newblock In {\em Proceedings of the {I}nternational {C}ongress of {M}athematicians. {V}olume {II}}, pages 125--145. Hindustan Book Agency, New Delhi, 2010.

\bibitem{Gantmacher1935}
F.~Gantmacher and M.~Krein.
\newblock {Sur les matrices oscillatoires}.
\newblock {\em C. R. Acad. Sci., Paris}, 201:577--579, 1935.

\bibitem{Gantmakher1937}
F.~Gantmakher and M.~Krein.
\newblock {Sur les matrices completement non n{\'e}}gatives et oscillatoires.
\newblock {\em Compos. Math.}, 4:445--476, 1937.

\bibitem{Goulden1992a}
I.~P. Goulden and D.~M. Jackson.
\newblock {Immanants of combinatorial matrices}.
\newblock {\em J. Algebra}, 148(2):305--324, 1992.

\bibitem{Greene1992}
C.~Greene.
\newblock {Proof of a conjecture on immanants of the {J}}acobi-{T}rudi matrix.
\newblock {\em Linear Algebra Appl.}, 171:65--79, 1992.

\bibitem{Haiman1993}
M.~Haiman.
\newblock {Hecke algebra characters and immanant conjectures}.
\newblock {\em J. Amer. Math. Soc.}, 6(3):569--595, 1993.

\bibitem{Hikita2024}
T.~Hikita.
\newblock {A proof of the Stanley-Stembridge conjecture}.
\newblock {\it Preprint}, \href{https://arxiv.org/abs/2410.12758}{arXiv:2410.12758}, 2024.

\bibitem{Malo1895}
E.~Mal{\'o}.
\newblock Note sur les {\'e}quations alg{\'e}briques dont toutes les racines sont r{\'e}elles.
\newblock {\em J. Math. Sp{\'e}c}, 4:7--10, 1895.

\bibitem{Nguyen2025}
S.~Nguyen and P.~Pylyavskyy.
\newblock Temperley–{L}ieb crystals.
\newblock {\em International Mathematics Research Notices}, 2025(7):rnaf080, 04 2025.

\bibitem{Rhoades2005}
B.~Rhoades and M.~Skandera.
\newblock Temperley-{L}ieb immanants.
\newblock {\em Ann. Comb.}, 9(4):451--494, 2005.

\bibitem{RHOADES2006793}
B.~Rhoades and M.~Skandera.
\newblock Kazhdan–{L}usztig immanants and products of matrix minors.
\newblock {\em Journal of Algebra}, 304(2):793--811, 2006.

\bibitem{Schur1918}
I.~Schur.
\newblock {\"U}ber endliche {Gruppen} und \emph{Hermite}sche {Formen}.
\newblock {\em Math. Z.}, 1:184--207, 1918.

\bibitem{Schur1914}
J.~Schur.
\newblock Zwei sätze über algebraische gleichungen mit lauter reellen wurzeln.
\newblock {\em Journal für die reine und angewandte Mathematik}, 144:75--88, 1914.

\bibitem{Sokal2024}
A.~Sokal.
\newblock Some positivity conjectures for symmetric functions motivated by classical theorems from the analytic theory of polynomials.
\newblock \href{https://drive.google.com/file/d/14AD7Ai1O2_h5ViwisT51rBFYe0LvjLH3/view?usp=sharing}{Talk} at The Many Combinatorial Legacies of Richard P. Stanley: Immense Birthday Glory of the Epic Catalonian Rascal, June 3, 2024 - June 7, 2024.

\bibitem{Stanley2024}
R.~Stanley.
\newblock An $m$-positivity conjecture related to bivariate {J}acobi--{T}rudi matrices.
\newblock {\it A posting at \href{https://mathoverflow.net/questions/475478/an-m-positivity-conjecture-related-to-bivariate-jacobi-trudi-matrices}{mathoverflow}, 2024}.

\bibitem{Stanley1982}
R.~P. Stanley.
\newblock {Some aspects of groups acting on finite posets}.
\newblock {\em J. Combin. Theory Ser. A}, 32(2):132--161, 1982.

\bibitem{EC2}
R.~P. Stanley.
\newblock {\em Enumerative combinatorics. {V}ol. 2}, volume 208 of {\em Cambridge Studies in Advanced Mathematics}.
\newblock Cambridge University Press, Cambridge, second edition, 2024.
\newblock With an appendix by Sergey Fomin.

\bibitem{Stanley1993}
R.~P. Stanley and J.~R. Stembridge.
\newblock {On immanants of {J}}acobi-{T}rudi matrices and permutations with restricted position.
\newblock {\em J. Combin. Theory Ser. A}, 62(2):261--279, 1993.

\bibitem{Stembridge1991}
J.~R. Stembridge.
\newblock {Immanants of totally positive matrices are nonnegative}.
\newblock {\em Bull. London Math. Soc.}, 23(5):422--428, 1991.

\bibitem{Stembridge1992a}
J.~R. Stembridge.
\newblock {Some conjectures for immanants}.
\newblock {\em Canad. J. Math.}, 44(5):1079--1099, 1992.

\end{thebibliography}

\end{document}